\documentclass[11pt, reqno]{amsart} 
\usepackage[utf8]{inputenc}
\usepackage[T1]{fontenc}

\usepackage{amssymb, amsmath, amsthm}
\usepackage[bookmarks, bookmarksdepth=2, colorlinks=true, linkcolor=blue, citecolor=blue, urlcolor=blue]{hyperref}

\usepackage[alphabetic,lite]{amsrefs}
\usepackage{verbatim}
\usepackage{amscd}   
\usepackage[all]{xy} 
\usepackage{youngtab} 
\usepackage{young} 
\usepackage{ytableau}
\usepackage{tikz}
\usepackage{ mathrsfs }
\usepackage{cases}
\usepackage{array}
\usepackage{cellspace}
\usepackage{calligra,mathrsfs}
\usepackage{bm}
\usepackage{graphicx}
\usepackage{rank-2-roots}
\usepackage{float}
\usepackage{enumitem}

\newcommand{\defi}[1]{{\bf\upshape\sffamily #1}}

\DeclareMathOperator{\ShHom}{\mathscr{H}\text{\kern -3pt {\calligra\large om}}\,}

\renewcommand{\a}{\alpha}
\renewcommand{\b}{\beta}
\newcommand{\bw}{\bigwedge}
\newcommand{\bB}{{\bf B}}

\newcommand{\bg}{{\bf g}}

\newcommand{\bN}{{\bf N}}

\newcommand{\bG}{{\bf G}}
\newcommand{\bh}{{\bf h}}

\def\kk{{\mathbf k}}
\renewcommand{\ll}{\lambda}

\newcommand{\oo}{\otimes}

\newcommand{\GL}{\operatorname{GL}}

\newcommand{\rk}{\operatorname{rank}}

\newcommand{\Sym}{\operatorname{Sym}}
\newcommand{\Tor}{\operatorname{Tor}}

\renewcommand{\det}{\operatorname{det}}

\renewcommand{\ker}{\operatorname{ker}}

\newcommand{\bb}[1]{\mathbb{#1}}

\renewcommand{\rm}[1]{\textrm{#1}}
\newcommand{\mc}[1]{\mathcal{#1}}
\newcommand{\mf}[1]{\mathfrak{#1}}
\newcommand{\ol}[1]{\overline{#1}}
\newcommand{\op}[1]{\operatorname{#1}}

\newcommand{\ul}[1]{\underline{#1}}

\def\PP{{\mathbf P}}
\def\lra{\longrightarrow}

\newtheorem{theorem}{Theorem}[section]
\newtheorem*{theorem*}{Theorem}
\newtheorem*{problem*}{Problem}
\newtheorem{lemma}[theorem]{Lemma}

\newtheorem{proposition}[theorem]{Proposition}
\newtheorem{corollary}[theorem]{Corollary}
\newtheorem*{corollary*}{Corollary}

\newtheorem*{main-thm*}{Main Theorem}
\newtheorem*{linear-resolutions*}{Theorem on Linear Resolutions}
\newtheorem*{regularity-powers*}{Theorem on Regularity}
\newtheorem*{injectivity-Ext*}{Theorem on Injectivity of Maps of Ext Modules}
\newtheorem*{Kodaira*}{Kodaira Vanishing for Determinantal Thickenings}

\theoremstyle{definition}

\newtheorem*{definition*}{Definition}
\newtheorem{example}[theorem]{Example}

\newtheorem{problem}[theorem]{Problem}

\theoremstyle{remark}
\newtheorem{remark}[theorem]{Remark}
\newtheorem*{remark*}{Remark}

\numberwithin{equation}{section}

\begin{document}

\title[Cohomology on the incidence correspondence]{Cohomology characters on the incidence correspondence}

\author{Annet Kyomuhangi}
\address{Department of Mathematics, Busitema University, P.O. Box 236, Tororo}
\email{annet.kyomuhangi@gmail.com}

\author{Emanuela Marangone}
\address{Dipartimento di Matematica, Università degli Studi di Genova, Via Dodecaneso, 35, 16146 Genova, Italy \newline
\indent CIMAT - Centro de Investigación en Matemáticas, Valenciana, 36023 Guanajuato, Mexico}
\email{emanuela.marangone@cimat.mx}

\author{Claudiu Raicu}
\address{Department of Mathematics, University of Notre Dame, 255 Hurley, Notre Dame, IN 46556\newline
\indent Institute of Mathematics ``Simion Stoilow'' of the Romanian Academy}
\email{craicu@nd.edu}

\author{Ethan Reed}
\address{Academy of Mathematics and Systems Science, Chinese Academy of Sciences, No. 55 Zhongguancun
East Road, Beijing, 100190, China.}
\email{ethan.reed@amss.ac.cn}


\date{\today}


\begin{abstract} 
We investigate the cohomology of line bundles on the incidence correspondence, the partial flag variety parametrizing pairs consisting of a point in projective space and a hyperplane containing it. In characteristic zero, this cohomology is governed by the Borel--Weil--Bott theorem. In characteristic $p>0$, however, it becomes considerably subtler, and admits an equivalent reformulation in terms of cohomology tables for divided powers of the cotangent bundle on projective space. Our approach to the problem involves passing to infinitesimal thickenings of the incidence correspondence inside the ambient product of projective spaces. This leads to recursive formulas for the cohomology, generalizing earlier work of Donkin, of Liu, and of Gao--Raicu. We obtain generating functions for cohomology characters, expressed using truncated Schur polynomials and symmetric polynomials encoding the higher structure constants of the Verlinde algebras of $\op{SU}(2)$ at levels $(p-2)$ and $(2p-2)$. Along the way, we exploit two important connections with multiplication in the graded Green--Han--Monsky representation ring: one relates this ring to cohomology, and another connects to the Verlinde algebras through the work of Coulembier--Etingof--Ostrik.
\end{abstract}

\maketitle

\section{Introduction}\label{sec:intro}

This paper is motivated by two fundamental cohomological questions that are settled in characteristic zero by the Borel--Weil--Bott theorem, but remain much subtler in positive characteristic:

\begin{problem}\label{pb:coh-flag}
    Describe the cohomology groups of line bundles on flag varieties.
\end{problem}

\begin{problem}\label{pb:poly-coh-tables}
    Given a polynomial functor $\mathcal{P}$, describe the cohomology table of $\mathcal{P}(\Omega)$, where $\Omega$ denotes the cotangent bundle on projective space.
\end{problem}

\noindent Problem~\ref{pb:coh-flag} is so fundamental that any brief attempt at motivation would do it little justice; instead, we refer the reader to the survey \cite{andersen-survey}, which contains many pointers to important work across a vast literature. Problem~\ref{pb:poly-coh-tables}, together with its connections to flag varieties and modular representation theory, is explored in depth in \cites{rai-vdb1,rai-vdb2}. Even in characteristic zero, when $\mathcal{P}=\mathbb{S}_{\lambda}$ is a Schur functor, the solution to Problem~\ref{pb:poly-coh-tables} already had an important application to Boij--S\"oderberg theory via the equivariant construction of supernatural vector bundles on projective space \cite{ESW}*{Theorem~5.6}, \cite{ES}*{Theorem~6.2}. When $\mathcal{P}=\Sym^d$ is a symmetric power functor, Problem~\ref{pb:poly-coh-tables} is also connected to vanishing theorems for Koszul modules, which are a crucial ingredient in establishing Green's conjecture for general curves in sufficiently positive characteristic \cite{AFPRW}*{Theorem~1.3}, \cite{rai-vdb1}*{Theorem~7.1}.

The goal of this paper is to solve Problems~\ref{pb:coh-flag} and~\ref{pb:poly-coh-tables} in a special case, in all positive characteristics: for the \defi{incidence correspondence}, the partial flag variety parametrizing pairs consisting of a point in projective space and a hyperplane containing it, and for the \defi{divided power functor} $\mc{P}=D^d$, defined as the subspace of $d$-fold tensors invariant under the permutation action of the symmetric group $\mathfrak{S}_d$,
\[D^dV = \left(V^{\oo d}\right)^{\mf{S}_d}.\] 
This substantially extends \cite{gao-raicu}, which only provides a characterization of the vanishing and nonvanishing behavior of cohomology. To our knowledge, this gives the first non-trivial infinite family of flag varieties of Picard rank~$2$ where Problem~\ref{pb:coh-flag} is resolved in every characteristic. As for Problem~\ref{pb:poly-coh-tables}, the only other family of functors for which it is resolved consists of exterior powers $\bw^d$, for which the cohomology is the same in every characteristic. Even the particular setting we consider already reveals the richness of these problems, leading to unexpected connections with questions of independent interest, such as the determination of the splitting type of vector bundles of principal parts on the projective line, the study of multiplication in the graded Green--Han--Monsky representation ring and the Verlinde algebras for $\op{SU}(2)$, and the Weak Lefschetz Property for monomial complete intersections. The results established here admit an algorithmic implementation, which we carry out in the computer algebra system Macaulay2 \cite{KMRR-M2}. They also lead to a characterization of the Weak Lefschetz Property for monomial complete intersections, an application we pursue in \cite{KMRR-WLP}.

\smallskip

\noindent {\bf The main objects.} We let $V=\kk^n$ where $\kk$ is an algebraically closed field of characteristic $p>0$, $n\geq 2$, let $S=\Sym(V) \simeq \kk[x_1,\cdots,x_n]$ be the symmetric algebra of $V$, and write $\PP=\bb{P}V=\op{Proj}(S)$ for the projective space parametrizing $1$-dimensional quotients of $V$. It carries the tautological exact sequence
\begin{equation}\label{eq:ses-on-PV}
0 \lra \mc{R} \lra V \oo \mc{O}_{\PP} \lra \mc{O}_{\PP}(1) \lra 0,
\end{equation}
where $\mc{R}$ denotes the tautological rank $(n-1)$ subbundle of $V \oo \mc{O}_{\PP}$. After twisting by $\mc{O}_{\PP}(-1)$, we obtain the Euler sequence, and in particular we identify the cotangent bundle on $\PP$ as $\Omega=\mc{R}(-1)$. The special case of Problem~\ref{pb:poly-coh-tables} considered in this paper is the determination of the sheaf cohomology groups of all twists
\[ D^d\mc{R} \oo \mc{O}_{\PP}(e) = D^d\Omega \oo \mc{O}_{\PP}(d+e)\quad\text{ for }d\geq 0,\ e\in\bb{Z},\]
referred to as the \defi{cohomology table} of $D^d\Omega$. Consider next the \defi{incidence correspondence}
\[ X = \{(p,H) : p\in H\} \subset \PP \times \PP^{\vee},\]
which is the hypersurface cut out by the bilinear form
\begin{equation}\label{eq:def-omega}
\omega = x_1y_1+\cdots+x_ny_n.
\end{equation}
The line bundles on $X$ are obtained by restriction from $\PP \times \PP^{\vee}$, and we write
\[ \mc{O}_X(a,b) = \mc{O}_{\PP \times \PP^{\vee}}(a,b)_{|_X}\quad\text{ for }(a,b)\in\bb{Z}^2.\]
The special case of Problem~\ref{pb:coh-flag} considered in this paper is the determination of the sheaf cohomology groups of all such line bundles $\mc{O}_X(a,b)$. The connection with the preceding problem is provided by the isomorphisms
\[H^i\left(\PP, D^d\Omega\oo \mc{O}_{\PP}(d+e)\right) = H^i\left(\PP, D^d\mc{R}\oo \mc{O}_{\PP}(e)\right) = H^{i+n-2}\left(X,\mc{O}_X(e+1,-d-n+1)\right) \oo \bw^n V^{\vee}.\]
While the right-hand side does not account for every line bundle on $X$, the elementary considerations recalled in Section~\ref{sec:prelim} show that it captures them up to natural symmetries. This forms the starting point of \cite{gao-raicu}, where a complete characterization of the vanishing and nonvanishing of the cohomology is obtained.

Our goal is to go beyond vanishing and give an exact description of the cohomology. To this end, we exploit its structure as a representation of $\GL(V)\simeq\GL_n$, and in particular of the maximal torus $(\kk^{\times})^n\subset \GL_n$. Recall that every finite dimensional representation $W$ of the algebraic torus $(\kk^{\times})^n$ has an eigenspace decomposition (or equivalently a $\bb{Z}^n$-grading), and we define the \defi{character} of $W$ to be the Laurent polynomial
\[[W] := \sum_{(i_1,\cdots,i_n)\in\bb{Z}^n} \dim\left(W_{(i_1,\cdots,i_n)}\right)\cdot z_1^{i_1}\cdots z_n^{i_n} \in \bb{Z}[z_1^{\pm 1},\cdots,z_n^{\pm 1}].\]
If $W$ is a $\GL_n$-representation, then $[W]$ is invariant under the action of $\mf{S}_n$ by coordinate permutations, that is, it belongs to the ring of symmetric Laurent polynomials (the \defi{character ring})
\begin{equation}\label{eq:def-Lambda} 
\Lambda = \bb{Z}[z_1^{\pm 1},\cdots,z_n^{\pm 1}]^{\mf{S}_n}.
\end{equation}
The main goal of our paper is then to understand the characters
\begin{equation}\label{eq:def-coh-chars} h^i(D^d\mc{R}(e)) = h^i(D^d\Omega(d+e)) := \left[H^i(\PP,D^d\Omega\oo \mc{O}_{\PP}(d+e))\right]\text{ and } h^i(\mc{O}_X(a,b)) := \left[H^i(X,\mc{O}_X(a,b))\right].
\end{equation}

\smallskip

\noindent {\bf The main ideas.}
A central insight of our approach is that the recursive character formulas arise naturally from the study of line bundle cohomology on the infinitesimal thickenings
$X^{(r)}\subset \PP\times\PP^{\vee}$. For suitable parameters $a,b$, one can show the existence of a torus-equivariant splitting
\begin{equation}\label{eq:splitting-HiOXab}
H^j\left(X^{(p)},\mc{O}_{X^{(p)}}(a,b)\right)
\simeq
\bigoplus_{i=0}^{p-1}
H^j\left(X,\mc{O}_X(a-i,b-i)\right).
\end{equation}
Since $\mc{O}_{X^{(p)}}$ is the pullback of $\mc{O}_X$ under the Frobenius morphism of $\PP\times\PP^{\vee}$, the cohomology of $\mc{O}_{X^{(p)}}(a,b)$ is related to Frobenius twists of the cohomology of line bundles $\mc{O}_X(a',b')$, where $a'$ and $b'$ are smaller than $a$ and $b$ by a factor of $p$; this is the geometric source of the recursion.

For arbitrary $a,b$, however, the splitting \eqref{eq:splitting-HiOXab} must be corrected to account for the possible non-vanishing of the connecting homomorphisms in cohomology. Analyzing these maps one weight at a time leads to a purely algebraic formulation. Thus, beyond the geometric intuition discussed in Section~\ref{sec:prelim}, the formal proof no longer requires infinitesimal thickenings: it reduces to studying the Artinian algebras
\begin{equation}\label{eq:defA}
A=\kk[T_1,\ldots,T_n]/
\langle T_1^{a_1},\ldots,T_n^{a_n}\rangle
\end{equation}
as $\kk[T]$-modules, where $T=T_1+\cdots+T_n$. In this formulation, the failure of the naive splitting is measured entirely by the structure of these modules, rather than by any additional geometric input. The resulting module-theoretic problem is encoded by $n$-fold products in the Green--Han--Monsky (GHM) representation ring (see below for a definition). The decisive point is that the ring structure reduces the analysis for arbitrary $n$ to the case $n=2$, while all the relevant information is already contained in a suitable quotient (the reduced GHM ring). This quotient admits a realization as a tensor product of Verlinde algebras for $\op{SU}(2)$ at levels $(p-2)$ and $(2p-2)$. Consequently, the entire correction to \eqref{eq:splitting-HiOXab} for general parameters $a,b$ is governed by higher structure constants in the reduced GHM ring, or equivalently in the corresponding Verlinde algebras.

\medskip

\noindent{\bf Generating functions for cohomology.} To answer Problem~\ref{pb:poly-coh-tables} for $\mc{P}=D^d$, we focus on the first cohomology group, and consider the generating function $\bG(u,v)\in\Lambda[[u,v]]$, defined by
\begin{equation}\label{eq:def-Guv}
    \bG(u,v) = \sum_{d\geq 0,\ e\geq -1} h^1(D^d\Omega(d+e))\cdot u^d\cdot v^{d+e}.
\end{equation}
It can be shown that $h^0(D^d\Omega(d+e))=h^1(D^{e+1}\Omega(d+e))$, that $h^i(D^d\Omega(d+e))=0$ for $i\neq 0,1,n-1$, and that $h^{n-1}(D^d\Omega(d+e))$ does not depend on the characteristic of the underlying field. Accordingly, a complete answer to Problem~\ref{pb:poly-coh-tables} reduces to determining $\bG(u,v)$. To this end, we introduce some notation. We define
\[ E(\ul{z}) = \prod_{i=1}^n(1+z_i)=\sum_{i=0}^n e_i(\ul{z}),\]
where $e_i$ denotes the $i$-th \defi{elementary symmetric polynomial}. We set
\begin{equation}\label{eq:def-thetar-Psir}
\theta_r = \frac{r\pi}{p}\quad\text{and}\quad\Psi_r(\ul{z}) = \prod_{i=1}^n\frac{1-(-1)^rz_i^p}{1-2\cos(\theta_r)z_i+z_i^2}\quad\text{ for }r=1,\cdots,p.
\end{equation}
We write $\delta_{r,p}$ for the \defi{Kronecker delta}, and let
\begin{equation}\label{eq:def-Thetaz-Theta'z}
\Theta(\ul{z}) = E(\ul{z})\cdot\sum_{r=1}^p \frac{2-\delta_{r,p}}{p}\sin^2\left(\frac{\theta_r}{2}\right)\Psi_r(\ul{z}),\quad \Theta'(\ul{z}) = E(\ul{z})\cdot\sum_{r=1}^p (-1)^{r+1}\frac{2-\delta_{r,p}}{p}\sin^2\left(\frac{\theta_r}{2}\right)\Psi_r(\ul{z}).
\end{equation}
Since $\Theta(\ul{z})$ and $\Theta'(\ul{z})$ are nonhomogeneous symmetric polynomials, we can decompose them into their even- and odd-degree parts as $\Theta(\ul{z}) = \Theta_{even}(\ul{z})+\Theta_{odd}(\ul{z})$, $\Theta'(\ul{z}) = \Theta'_{even}(\ul{z})+\Theta'_{odd}(\ul{z})$. For each $q=p^s$, $s\geq 0$, we consider the endomorphism 
$F^q:\Lambda\lra\Lambda$ defined by $F^q(z_i)=z_i^q$, and extend it to $\Lambda[[t]]$, $\Lambda[[u,v]]$ by setting $F^q(t)=t^q$, $F^q(u)=u^q$, $F^q(v) = v^q$. We then define
\begin{equation}\label{eq:def-Mz-Mt} 
B(\ul{z}) = \begin{bmatrix}
    \Theta_{even}(\ul{z}) & \Theta'_{odd}(\ul{z}) \\
    \Theta_{odd}(\ul{z}) & \Theta'_{even}(\ul{z}) \\
\end{bmatrix},\text{ and let }\bB(t) = F^{p}(B(t\ul{z}))\cdot F^{p^2}(B(t\ul{z}))\cdots = \prod_{s\geq 1}F^{p^s}(B(t\ul{z})),
\end{equation}
where the infinite product converges to a $2\times 2$ matrix with entries in $\Lambda[[t]]$. We also define
\begin{equation}\label{eq:def-Acz} A_c(\ul{z}) = \frac{2}{p}\sum_{r=1}^{p-1}\sin(\theta_r)\sin(c\theta_r)\Psi_r(\ul{z})\quad\text{ for }c=1,\cdots,p-1.
\end{equation}
While it is not immediately apparent from their definition, the polynomials $A_c(\ul{z})$ and $\Theta(\ul{z})$, $\Theta'(\ul{z})$ have non-negative integer coefficients. These coefficients are described in Section~\ref{subsec:Ver-structure-constants} as higher structure constants for Verlinde algebras. Remarkably, we show in \eqref{eq:Ac-from-lmu} that each homogeneous component of $A_c(\ul{z})$ encodes a simple modular character of $\GL_n$, whose highest weight has the form $\ll=(p-2,p-2,\cdots,p-2,u,v)$, and in particular it is \defi{$(p-2)$-special} in the sense of \cite{mat-pap}. Therefore each $A_c(\ul{z})$ admits an alternative description involving tableau combinatorics (see Section~\ref{subsec:Verlinde-from-tableaux}). Although the homogeneous components of $\Theta(\ul{z})$ and $\Theta'(\ul{z})$ are no longer simple characters, they can be computed in terms of $A_c(\ul{z})$ and elementary symmetric polynomials.

We next consider the power series $\bN_c(t)\in\Lambda[[t]]$ defined by
\begin{equation}\label{eq:Nct=AM10product}
\bN_c(t) = \begin{bmatrix}
    A_c(t\ul{z}) & A_{p-c}(t\ul{z})
\end{bmatrix}\cdot
\bB(t)\cdot
\begin{bmatrix}
    1 \\ 0
\end{bmatrix}
\quad\text{ for }c=1,\cdots,p-1,
\end{equation}
and assemble them into a bivariate series $\bN(u,v)\in\Lambda[[u,v]]$ defined by
\begin{equation}\label{eq:def-Nuv}
    \bN(u,v) = \sum_{c=1}^{p-1}\bN_c\left(u^{1/2}v\right)\cdot u^{1/2}\cdot\frac{u^{c/2}-u^{p-c/2}}{1-u^p}.
\end{equation}
Finally, we let
\begin{equation}\label{eq:def-bhqt}
\bh^{(q)}(t) = \prod_{i=1}^n\frac{1-F^q(tz_i)}{1-tz_i} = \prod_{i=1}^n(1+tz_i+t^2z_i^2+\cdots+t^{q-1}z_i^{q-1}),
\end{equation}
noting that $\bh^{(q)}(t) = \sum_{d\geq 0}h^{(q)}_d(\ul{z})\cdot t^d$, with coefficients the \defi{$q$-truncated complete symmetric polynomials}
\begin{equation}\label{eq:def-hqd}
h^{(q)}_d(\ul{z}) = \sum_{\substack{i_1+\cdots+i_n = d \\ 0\leq i_j<q}}z_1^{i_1}\cdots z_n^{i_n}.
\end{equation}

\begin{theorem}\label{thm:main-nonrec-coh} If we let $\bG(u,v)$, $\bN(u,v)$ as in \eqref{eq:def-Guv}, \eqref{eq:def-Nuv} then $\bG$ satisfies the functional equation
    \begin{equation}\label{eq:func-eqn-Guv}
    \bG(u,v) = \frac{\bh^{(p)}(uv)\cdot \bh^{(p)}(v)}{1+u+\cdots+u^{p-1}}\cdot F^p(\bG(u,v)) + \bN(u,v).
    \end{equation}
    If we write $q=p^s$ for $s\geq 0$, then
    \[\bG(u,v) = \sum_{q\geq 1} \frac{\bh^{(q)}(uv)\cdot \bh^{(q)}(v)}{1+u+\cdots+u^{q-1}}\cdot F^q\left(\bN(u,v)\right). \]
\end{theorem}

To elucidate the notation, we consider the special case $p=2$, where $\theta_1=\pi/2$, $\theta_2=\pi$, and
\[\Psi_1(\ul{z})=1,\quad \Psi_2(\ul{z}) = \prod_{i=1}^n\frac{1-z_i}{1+z_i}.\]
It follows that
\begin{equation}\label{eq:Thetaz-p=2}
\Theta(\ul{z}) = E(\ul{z})\cdot\frac{\Psi_1(\ul{z})+\Psi_2(\ul{z})}{2}=\frac{\prod_{i=1}^n(1+z_i) + \prod_{i=1}^n(1-z_i)}{2} = \sum_{i\geq 0}e_{2i}(\ul{z}) = \Theta_{even}(\ul{z}).
\end{equation}
Similarly, we get $\Theta'(\ul{z}) = \sum_{i\geq 0}e_{2i+1}(\ul{z}) = \Theta'_{odd}(\ul{z})$ and therefore
\begin{equation}\label{eq:Mz-charp=2}
B(\ul{z}) = \begin{bmatrix}
    \Theta(\ul{z}) & \Theta'(\ul{z}) \\
    0 & 0 \\
\end{bmatrix}.
\end{equation}
Since $A_1(\ul{z})=\Psi_1(\ul{z})=1$, it follows that
\begin{equation}\label{eq:N1t-char2}
\bN_1(t) = \prod_{s\geq 1}F^{2^s}(\Theta(t\ul{z})) \quad\text{and}\quad \bN(u,v) = \bN_1(u^{1/2}v)\cdot\frac{u}{1+u}. \end{equation}
We can now reinterpret Theorem~\ref{thm:main-nonrec-coh} in terms of Nim symmetric polynomials, as follows (see also \cite{gao-raicu}*{Theorem~1.9}, \cite{GRV}*{Section~5.3}). For $d\geq 0$, we consider its $2$-adic expansion $d=(d_k\cdots d_0)_2$, where
\[ d= \sum_{i=0}^k d_i\cdot 2^i,\text{ with }d_i\in\{0,1\}\text{ for all }i.\]
We define the \defi{Nim-sum} $a\oplus b$ by adding the digits in their $2$-adic expansions modulo $2$:
\begin{equation}\label{eq:def-nim-sum} a\oplus b = c\text{ if and only if }a_i + b_i \equiv c_i \text{ mod }2\text{ for all }i.
\end{equation}
We then define the \defi{Nim symmetric polynomials} via
\begin{equation}\label{eq:def-Nim-pol}
 \mc{N}_m(\ul{z}) = \sum_{\substack{i_1+\cdots+i_n=2m \\ i_1\oplus i_2\oplus\cdots\oplus i_n = 0}} z_1^{i_1} z_2^{i_2}\cdots z_n^{i_n}.
\end{equation}
If we define the \defi{$q$-truncated Schur polynomials $s^{(q)}_{(a,b)}$} by
\begin{equation}\label{eq:def-sqab}
s^{(q)}_{(a,b)} = h^{(q)}_a\cdot h^{(q)}_b-h^{(q)}_{a+1}\cdot h^{(q)}_{b-1}
\end{equation}
then Theorem~\ref{thm:main-nonrec-coh} can be shown to yield the following consequence in characteristic $p=2$.

\begin{theorem}\label{thm:h1coh-char2-nonrecursive}
    If $\op{char}(\kk)=2$ and $e\geq d-1$ then
 \[ h^1(D^d\mc{R}(e)) = \sum_{(q,m,j)\in\Lambda_d} F^{2q}(\mc{N}_m)\cdot s^{(q)}_{(e-(2m-2j-1)q,d-(2m+2j+1)q)},\]
 where $\Lambda_d = \left\{(q,m,j) | q=2^s\text{ for some }s\geq 1,\ m,j\geq 0,\text{ and }(2m+2j+1)q\leq d\right\}$.
\end{theorem}

Notice that Theorem~\ref{thm:h1coh-char2-nonrecursive} corrects the statement of \cite{GRV}*{Conjecture~5.2}. Revisiting the example from \cite{GRV}*{Section~5.3}, we get for $d=6$ and $e\geq 5$ that $\Lambda_d =  \{(4,0,0), (2,0,0), (2,0,1), (2,1,0)\}$, hence
\[h^1(D^6\mc{R}(e)) = s^{(4)}_{(e+4,2)} + s^{(2)}_{(e+2,4)} + s^{(2)}_{(e+6,0)} + F^4(\mc{N}_1) \cdot s^{(2)}_{(e-2,0)}.\]
The term $s^{(2)}_{(e+6,0)}$, which is just the elementary symmetric polynomial of degree $e+6$, was missing from the original conjecture, and the first time it occurs is when the number of variables is $n\geq e+6 \geq 11$. This explains why the computational verification of the example in loc. cit. was correct up to $n=10$ variables.


\medskip

\noindent{\bf Recursive character formulas.} The functional equation in Theorem~\ref{thm:main-nonrec-coh} yields a recursive procedure for computing cohomology, but it does so through the coefficients of the power series $\bN(u,v)$, which are themselves rather intricate. However, a short argument allows one to recast the recursion entirely in terms of Schur and truncated Schur polynomials, as we explain next. We recall \eqref{eq:def-sqab} and define for $d\geq 0$ and $e\in\bb{Z}$ 
\begin{equation}\label{eq:def-Phi-de} 
\Phi_{d,e} = \sum_{j\geq 0} s^{(p)}_{(e+jp,d-jp)}.
\end{equation}
Observing that in the power series $\bN(u,v)$, the coefficient of $u^d v^{d+e}$ vanishes whenever $e\geq d-1$, we obtain the following, using the notation of \eqref{eq:def-coh-chars} and \eqref{eq:def-Phi-de}.

\begin{theorem}\label{thm:coh-recursion}
 If $e\geq d-1$ then we have for $i=0,1$ that
 \[ h^i(D^d\mc{R}(e)) = \sum_{a=0}^{\lfloor \frac{d}{p}\rfloor}\sum_{b=-1}^{\lfloor \frac{d+e}{p}\rfloor} \Phi_{d-ap,e-bp}\cdot F^p\left(h^i(D^a\mc{R}(b))\right).\]
\end{theorem}

Since no assumption $b\geq a-1$ is imposed on the right-hand side of the equation above, the formula may not at first appear to furnish a recursion. However, the recursive nature of Theorem~\ref{thm:coh-recursion} follows from \eqref{eq:coh-P-vs-Pdual}, which allows us to restrict when convenient to $e\geq d-1$ (see also Theorem~\ref{thm:h1coh-char2-nonrecursive}). Moreover, if we consider the \defi{complete symmetric polynomials $h_d$}, and the \defi{Schur polynomials $s_{(a,b)}$}, defined~by
\begin{equation}\label{eq:def-hd-sab}
h_d = \sum_{\substack{i_1+\cdots+i_n = d \\ i_j\geq 0}}z_1^{i_1}\cdots z_n^{i_n}\quad\text{and}\quad s_{(a,b)} = h_a\cdot h_b-h_{a+1}\cdot h_{b-1},
\end{equation}
then for $d\geq 0$, $e\geq -1$ we have
\begin{equation}\label{eq:chi-DdRe=sed}
h^0(D^d\mc{R}(e)) - h^1(D^d\mc{R}(e)) = s_{(e,d)}.
\end{equation}
If we assume further that $d<p$, the cohomology is given  (as in characteristic zero) by
\begin{equation}\label{eq:h01-for-small-d} 
h^0(D^d\mc{R}(e)) = \begin{cases} 
s_{(e,d)} & \text{if }e\geq d \\
0 & \text{if }e\leq d-1
\end{cases}
\quad\text{and}\quad
h^1(D^d\mc{R}(e)) = \begin{cases} 
s_{(d-1,e+1)} & \text{if }e\leq d-2 \\
0 & \text{if }e\geq d-1
\end{cases}
\end{equation}

\begin{example}\label{ex:small-rec-h1Ddr}
    Suppose that $n=5$, $p=2$. Using Theorem~\ref{thm:coh-recursion}, along with \eqref{eq:h01-for-small-d}, we get
    \[ 
    \begin{aligned}
    h^1(D^3\mc{R}(2)) &= \Phi_{1,4}\cdot F^p\left(h^1(\mc{R}(-1))\right) = \Phi_{1,4} \cdot F^p(s_{(0,0)}) = h^{(p)}_4\cdot h^{(p)}_1 - h^{(p)}_5\cdot h^{(p)}_0 \\
    &= \left(\sum z_1^2 z_2z_3z_4\right) + 4\cdot z_1z_2z_3z_4z_5,
    \end{aligned}
    \]
    where the above sum is over the $\mf{S}_5$-orbit of the monomial $z_1^2 z_2z_3z_4$. It follows that the cohomology group $H^1(\PP^4,D^3\mc{R}(2))$ is a vector space of dimension $24$. If instead we take $d=2$, $e=3$, then we get
    \[h^1(D^2\mc{R}(3)) = \Phi_{0,5}\cdot F^p\left(h^1(\mc{R}(-1))\right) = z_1z_2z_3z_4z_5,\]
    so $H^1(\PP^4,D^2\mc{R}(3))$ is $1$-dimensional. In characteristic $p=3$ we have
    \[h^1(D^3\mc{R}(2)) = \Phi_{0,5}\cdot F^p\left(h^1(\mc{R}(-1))\right) = h^{(3)}_5 = \left(\sum z_1^2 z_2^2z_3\right) + \left(\sum z_1^2 z_2z_3z_4\right) +  z_1z_2z_3z_4z_5\]
    which is the character of a $51$-dimensional representation, while $h^1(D^2\mc{R}(3))=0$. For all other characteristics, we have $h^1(D^3\mc{R}(2)) = h^1(D^2\mc{R}(3)) = 0$.
\end{example}

As an application of Theorem~\ref{thm:coh-recursion}, we prove \cite{GRV}*{Conjecture~5.1} in Theorem~\ref{thm:small-weights}.

\begin{remark}\label{rem:n=2}
When $n=2$, we have that $V=\kk^2$ and $X=\PP^1$ is the projective line. Under this identification, $\mc{O}_X(1,0)=\mc{O}_{\PP^1}(1)$ and $\mc{O}_X(0,1)=\bw^2 V^{\vee}\oo\mc{O}_{\PP^1}(1)$. Moreover, $\mc{R}=\bw^2 V\oo\mc{O}_{\PP^1}(-1)$ and \eqref{eq:h01-for-small-d} holds for all $d\geq 0$, $e\geq -1$. Theorems~\ref{thm:main-nonrec-coh}--\ref{thm:coh-recursion} apply as stated for $n=2$, and in fact this is the reason for imposing the condition $e\geq -1$ in \eqref{eq:def-Guv} (when $n\geq 3$ one has $h^1(D^d\mc{R}(e))=0$ for $e\leq -2$). Although the cohomology of line bundles on $\PP^1$ is classical, it is not entirely obvious that it fits into the uniform framework of  Theorems~\ref{thm:main-nonrec-coh}--\ref{thm:coh-recursion}.
\end{remark}

\noindent{\bf The graded Green--Han--Monsky representation ring.} Following \cite{green}, \cite{han-monsky}, we consider the category of finite length graded $\kk[T]$-modules $M$, and define the tensor product $M\oo_{\kk} N$ by letting $T$ act~via
\[ T \cdot (m\oo n) = Tm\oo n + m\oo Tn.\]
This induces a natural multiplication on the set of isomorphism classes, and we refer to the resulting ring $\Delta$ as the \defi{graded Green--Han--Monsky (GHM) representation ring}. We write $\delta_d$ for $\kk[T]/(T^d)$ (and for its class in~$\Delta$), and note that every indecomposable object in the category is isomorphic to $\delta_d(-j)$ for some $j\in\bb{Z}$, where $j$ denotes the degree of its cyclic generator. We are particularly interested in the $n$-fold product $\delta_{a_1}\cdots\delta_{a_n}$, which represents the algebra \eqref{eq:defA} viewed as a $\kk[T]$-module by letting $T=T_1+\cdots+T_n$. The following example shows that the multiplication depends in subtle ways on the characteristic $p$ of $\kk$: one has
\[ \delta_3\delta_5 = \begin{cases}
\delta_7 + \delta_4(-1) + \delta_4(-2) & \text{if }p = 2, \\
\delta_6 + \delta_6(-1) + \delta_3(-2) & \text{if }p = 3, \\
\delta_5 + \delta_5(-1) + \delta_5(-2) & \text{if }p = 5, \\
\delta_7 + \delta_5(-1) + \delta_3(-2) & \text{otherwise.} \\
\end{cases}
\]
Multiplication in $\Delta$ has many equivalent interpretations; in particular, it may be viewed as the problem of determining the Jordan canonical form of the tensor product of two Jordan blocks. From this perspective, in characteristic zero the resulting decomposition already appears in Littlewood’s classical text \cite{littlewood}*{Section~10.2}. We record some properties of multiplication in $\Delta$ in Section~\ref{sec:Han-Monsky}, where we prove in particular the following (see Proposition~\ref{prop:delcj-in-product} and Corollary~\ref{cor:prod-dela-odd-Nim}).

\begin{theorem}\label{thm:delcj-in-product}
    \begin{enumerate}
    \item If $p\nmid c$ and $\delta_c(-j)$ is a summand of $\delta_{a_1}\cdots\delta_{a_n}$ then
    \[ c+2j = a_1+\cdots+a_n - (n-1).\]
    \item If $p=2$ then $\delta_{2c+1}(-j)$ appears as a summand of $\delta_{2a_1+1}\delta_{2a_2+1}\cdots\delta_{2a_n+1}$ if and only if
    \[ c = a_1\oplus a_2\oplus\cdots\oplus a_n \quad\text{ and }\quad j = a_1+\cdots+a_n-c.\]
    \end{enumerate}
\end{theorem}

The significance of Theorem~\ref{thm:delcj-in-product} is explained in the inductive procedure that we employ to study cohomology. It turns out that if we fix a weight $\ul{a}$, the non-triviality of connecting homomorphisms in an appropriate long exact sequence in cohomology can be rephrased in terms of the presence of certain summands in $\delta_{a_1}\cdots\delta_{a_n}$ (see Lemma~\ref{lem:rank-partialM}). Notice that part (2) of Theorem~\ref{thm:delcj-in-product} implies that in characteristic $p=2$, $\delta_1$ appears as a summand (with some shift) in $\delta_{2a_1+1}\delta_{2a_2+1}\cdots\delta_{2a_n+1}$ if and only if $a_1\oplus a_2\oplus\cdots\oplus a_n=0$, which means that $z_1^{a_1}\cdots z_n^{a_n}$ appears as a term in a Nim polynomial of appropriate degree. This connection leads to Theorem~\ref{thm:h1coh-char2-nonrecursive} and is discussed in Section~\ref{sec:coh-characters}.

Theorem~\ref{thm:delcj-in-product} has two key features: the relevant summands $\delta_c$ have length coprime to $p$, and for these summands the degree shifts are uniquely determined. This leads us to consider the ungraded quotient $\ol{\Delta}^u$ where $\delta_c=\delta_c(-j)$, and $\delta_{pm}=0$ for all $m$, which we call the \defi{reduced GHM ring}. It follows from \cite{CEO}*{Section~5} that $\ol{\Delta}^u$ is an infinite tensor product of Verlinde algebras for $\op{SU}(2)$ \cite{verlinde}, and this structure (which we recall in Section~\ref{subsec:redGHM=tensor-Verlinde}) is then the source of the trigonometric character formulas for cohomology.

\medskip

\noindent {\bf Related work.} Prior to the present work, recursive formulas were known only in the case of the $3$-dimensional flag variety \cites{donkin,liu,gao-raicu}, for which non-recursive formulas were explored in \cite{odorney}. Some additional partial results for the incidence correspondence in higher dimensions were available in \cites{liu-flag1,liu-polo,gao-raicu}. A complete characterization of the vanishing and nonvanishing behavior of cohomology was obtained in \cite{gao-raicu}, extending the dimension $3$ results of Griffith \cite{griffith}. Partial results on the vanishing and nonvanishing of cohomology for flag varieties of type $B_2$ and $G_2$ may be found in \cite{andersen-b2}*{Section~5} and \cites{humphreys-G2,and-kan}. For a collection of open questions, several of which are resolved in this paper, we refer the reader to \cite{GRV}.

\medskip

\noindent{\bf Organization.} In Section~\ref{sec:prelim} we collect a number of elementary observations relating the cohomology of line bundles on the incidence correspondence, the cohomology tables of $D^d\Omega$, and the Artinian algebras~\eqref{eq:defA}. In Section~\ref{sec:Han-Monsky} we establish the necessary properties of multiplication in the GHM ring, and in particular verify Theorem~\ref{thm:delcj-in-product}. Section~\ref{sec:hom-kT} discusses the basic homological algebra behind understanding weight spaces for cohomology, and provides the main connection with the multiplication in the GHM ring. In Section~\ref{sec:Verlinde} we explain the link between GHM rings and Verlinde algebras, and introduce the symmetric polynomials encoding their higher structure constants. The final Section~\ref{sec:coh-characters} combines the prior results in order to obtain the computation of cohomology characters, proving Theorems~\ref{thm:main-nonrec-coh}, \ref{thm:h1coh-char2-nonrecursive}, and~\ref{thm:coh-recursion}, as well as \cite{GRV}*{Conjecture~5.1}.

\section{Preliminaries}\label{sec:prelim}

Throughout the paper we let $V=\kk^n$, $n\geq 2$, where $\kk$ is an algebraically closed field of characteristic $p>0$. This assumption is in place in order to avoid some technicalities (for instance, if we were working over a non-closed field, representations of $\GL_n$ would have to be interpreted as algebraic representations of the group scheme $\GL_n$), but since the cohomology groups we consider are defined over the prime field and commute with extensions of the base field, the results will only depend on the characteristic $p$. We let $S=\Sym(V)\simeq\kk[x_1,\cdots,x_n]$ and consider the projective space $\PP=\bb{P}V=\op{Proj}(S)$. We similarly set $\PP^{\vee} = \op{Proj}(\kk[y_1,\cdots,y_n])$, where the variables $y_i$ are dual to the $x_i$. The choice of basis on $V$ determines a maximal torus $(\kk^{\times})^n$ inside $\GL(V)\simeq\GL_n$ which acts with weight $\vec{e}_i$ on $x_i$, and weight $-\vec{e}_i$ on $y_i$, where $\vec{e}_i$ denotes the $i$-th standard unit vector. This action extends to all cohomology groups considered in this paper.

\subsection{The incidence correspondence and some basic cohomology characters.}
\label{subsec:inc-cor-coh}
We let
\[Y = \PP \times \PP^{\vee},\quad X \overset{\eqref{eq:def-omega}}{=} \{\omega=0\}\subset Y,\quad \mc{O}_X(m_1,m_2)=\mc{O}_Y(m_1,m_2)_{|_X}.\]
We are interested in the characters of the cohomology groups of $\mc{O}_X(m_1,m_2)$, and we start by recalling the well-known description for $\mc{O}_Y(m_1,m_2)$. With notation \eqref{eq:def-hd-sab}, the non-zero cohomology characters are:
\begin{itemize}
    \item If $m_1,m_2\geq 0$, then $h^0(\mc{O}_Y(m_1,m_2))=h_{m_1}(\ul{z})\cdot h_{m_2}(\ul{z}^{-1})$, where $\ul{z}^{-1}=(z_1^{-1},\cdots,z_n^{-1})$.
    \item If $m_1\geq 0$, $m_2\leq -n$ then $h^{n-1}(\mc{O}_Y(m_1,m_2))=h_{m_1}(\ul{z})\cdot h_{-m_2-n}(\ul{z})\cdot z_1\cdots z_n$.
    \item If $m_1\leq -n$, $m_2\geq 0$ then $h^{n-1}(\mc{O}_Y(m_1,m_2))=h_{-m_1-n}(\ul{z}^{-1})\cdot h_{m_2}(\ul{z}^{-1})\cdot z_1^{-1}\cdots z_n^{-1}$.
    \item If $m_1,m_2\leq -n$ then $h^{2n-2}(\mc{O}_Y(m_1,m_2))=h_{-m_1-n}(\ul{z}^{-1})\cdot h_{-m_2-n}(\ul{z})$.
\end{itemize}
We can then study the cohomology of $\mc{O}_X(m_1,m_2)$ using the long exact sequence associated with
\[0 \lra \mc{O}_Y(m_1-1,m_2-1) \lra \mc{O}_Y(m_1,m_2) \lra \mc{O}_X(m_1,m_2)\lra 0.\]
Indeed, it follows that (see also \cite{gao-raicu}*{Introduction}) 
\begin{itemize}
    \item If $m_1,m_2\geq 0$ then \[h^0(\mc{O}_X(m_1,m_2))=h_{m_1}(\ul{z})\cdot h_{m_2}(\ul{z}^{-1})-h_{m_1-1}(\ul{z})\cdot h_{m_2-1}(\ul{z}^{-1}).\]
    \item If $m_1,m_2\leq -n+1$ then 
    \begin{equation}\label{eq:h2n-3-DdRe}
    h^{2n-3}(\mc{O}_X(m_1,m_2))=h_{-m_1+1-n}(\ul{z}^{-1})\cdot h_{-m_2+1-n}(\ul{z})-h_{-m_1-n}(\ul{z}^{-1})\cdot h_{-m_2-n}(\ul{z}).
    \end{equation}
    \item In addition to the cases above, $h^i(\mc{O}_X(m_1,m_2))$ may only be non-zero for $i=n-2,n-1$, when $m_1\geq 0$ and $m_2\leq -n+1$, or $m_1\leq -n+1$ and $m_2\geq 0$. These are the interesting cohomology groups where dependence of the characteristic may occur. Recalling notation \eqref{eq:def-Lambda} and writing ${}^{\vee}:\Lambda\lra\Lambda$ for the involution that sends $z_i\lra z_i^{-1}$, we have
\[h^i(\mc{O}_X(m_1,m_2)) = h^i(\mc{O}_X(m_2,m_1))^{\vee}\]
so we may assume without loss of generality that $m_1\geq 0$, $m_2\leq -n+1$. 
\end{itemize}

\subsection{Infinitesimal thickenings and Frobenius.} 
\label{subsec:thickX}

We consider more generally the infinitesimal thickenings $X=X^{(1)}\subset X^{(2)}\subset\cdots\subset X^{(p)}$ of $X$ inside~$Y$, defined by the short exact sequences
\begin{equation}\label{eq:def-ses-Xr}
    0\lra \mc{O}_Y(-r,-r) \overset{\cdot\omega^r}{\lra} \mc{O}_Y \lra \mc{O}_{X^{(r)}} \lra 0.
\end{equation}
If we write $\mc{I}$ for the ideal sheaf of $X$ inside $X^{(p)}$ then $\mc{I}^{r-1}/\mc{I}^r=\mc{O}_X(-r+1,-r+1)$ is the ideal sheaf of $X^{r-1}$ inside $X^r$, and we have a short exact sequence
\[0\lra \mc{O}_X(m_1-r+1,m_2-r+1) \lra \mc{O}_{X^{(r)}}(m_1,m_2) \lra \mc{O}_{X^{(r-1)}}(m_1,m_2) \lra 0,\text{ for }r=2,\cdots,p.\]
It follows that $\mc{O}_{X^{(r)}}(m_1,m_2)$ has a filtration with composition factors $\mc{O}_X(m_1-s,m_2-s)$, for $s=0,\cdots,r-1$. In particular, the discussion in Section~\ref{subsec:inc-cor-coh} implies that if $m_1\geq r-n+1$, $m_2\leq -1$, then the cohomology of $\mc{O}_{X^{(r)}}(m_1,m_2)$ can only occur in degrees $n-2,n-1$, and we have an exact sequence
\begin{equation}\label{eq:les-coh-thick}
\begin{aligned}
   0 &\to H^{n-2}\left(\mc{O}_X(m_1-r+1,m_2-r+1)\right) \to H^{n-2}\left(\mc{O}_{X^{(r)}}(m_1,m_2)\right) \to H^{n-2}\left(\mc{O}_{X^{(r-1)}}(m_1,m_2)\right) \to  \\
   & \to H^{n-1}\left(\mc{O}_X(m_1-r+1,m_2-r+1)\right) \to H^{n-1}\left(\mc{O}_{X^{(r)}}(m_1,m_2)\right) \to H^{n-1}\left(\mc{O}_{X^{(r-1)}}(m_1,m_2)\right) \to 0
\end{aligned}
\end{equation}
We consider next
\begin{equation}\label{eq:defY-Y'-frob}
    Y' = \bb{P}(F^pV) \times \bb{P}(F^pV^{\vee}),\quad\text{and let }\varphi:Y \lra Y'\text{ be the Frobenius morphism},
\end{equation}
where $F^p$ denotes the Frobenius power subfunctor of $\Sym^p$. 
We can define by analogy the incidence correspondence $X'\subset Y'$, and note that the sequence \eqref{eq:def-ses-Xr} for $r=p$ is obtained as the pull-back under $\varphi$ for the defining sequence
\[0\lra \mc{O}_{Y'}(-1,-1) \lra \mc{O}_{Y'}\lra \mc{O}_{X'} \lra 0,\]
and in particular $\mc{O}_{X^{(p)}}=\varphi^*\mc{O}_{X'}$. It is easy to relate the cohomology characters on $X$ and $X'$,
\begin{equation}\label{eq:cohX-vs-X'}
    h^i(\mc{O}_{X'}(m_1,m_2)) = F^p\left(h^i(\mc{O}_{X}(m_1,m_2))\right)\quad\text{ for all }i,m_1,m_2,
\end{equation}
which combined with a complete understanding of the connecting homomorphism in \eqref{eq:les-coh-thick}, for each $r=2,\cdots,p$, will provide the key technical input for our work, leading to the main character formulas in this paper.

\subsection{Divided powers of the cotangent bundle}
\label{subsec:div-pows-coh}

Recalling the tautological sequence \eqref{eq:ses-on-PV} and the notation \eqref{eq:def-coh-chars}, one has (see \cite{gao-raicu}*{(2.12)}, \cite{GRV}*{Section~5}) that
\begin{equation}\label{eq:hiDdre=OXlb}
    h^i(D^d\mc{R}(e)) = h^i(D^d\Omega(d+e)) = h^{i+n-2}(\mc{O}_X(e+1,-d+1-n))\cdot z_1^{-1}\cdots z_n^{-1},\text{ for }d\geq 0,e\in\bb{Z}.
\end{equation}
If $e\geq -1$ (so that $m_1=e+1\geq 0$, $m_2=-d+1-n\leq -n+1$) these groups can only be non-zero for $i=0,1$. Combining this with Serre duality on $X$, where the canonical line bundle is $\omega_X=\mc{O}_X(1-n,1-n)$, we get
\begin{equation}\label{eq:coh-P-vs-Pdual}
h^i(D^d\mc{R}(e)) = h^{1-i}(D^{e+1}\mc{R}(d-1))\quad\text{for }i=0,1,\ d\geq 0,\ e\geq -1.
\end{equation}
Theorem~\ref{thm:main-nonrec-coh} is concerned with describing $h^1(D^d\mc{R}(e))$, which then determines $h^0(D^d\mc{R}(e))$ by \eqref{eq:coh-P-vs-Pdual}.

When $e<-1$, $h^i(D^d\mc{R}(e))$ can only be non-zero for $i=n-1$. In light of \eqref{eq:h2n-3-DdRe}, \eqref{eq:hiDdre=OXlb}, we have
\[h^{n-1}(D^d\mc{R}(e)) = \frac{h_{-e-n}(\ul{z}^{-1})\cdot h_{d}(\ul{z})-h_{-e-1-n}(\ul{z}^{-1})\cdot h_{d-1}(\ul{z})}{z_1\cdots z_n}.\]

\subsection{Concrete realization of cohomology}
\label{subsec:concrete-coh}

We consider the polynomial ring $R = \kk[x_1,\cdots,x_n,y_1,\cdots,y_n]$, and the local cohomology module $L=H^n_{(y_1,\cdots,y_n)}(R)$, which has a coordinate independent description as
\[L= \bigoplus_{\substack{m_1\geq 0 \\ m_2\leq -n}}H^{n-1}(Y,\mc{O}_{Y}(m_1,m_2)) = \bigoplus_{\substack{m_1\geq 0 \\ m_2\leq -n}} \Sym^{m_1}V \oo D^{-m_2-n}V \oo \bw^n V.\]
We write $L_{m_1,m_2}=H^{n-1}(Y,\mc{O}_{Y}(m_1,m_2))$, and for $m_1=e$, $m_2=-d-n$, $d,e\geq 0$, we get
\[L_{e,-d-n} = H^{n-1}(Y,\mc{O}_Y(e,-d-n)) = \bigoplus_{\substack{e_i,d_j\geq 0 \\ d_1+\cdots+d_n=d,\\ e_1+\cdots+e_n=e}} \kk\cdot\frac{x_1^{e_1}\cdots x_n^{e_n}}{y_1^{1+d_1}\cdots y_n^{1+d_n}}.\]
From the long exact sequence associated to (an appropriate twist of) \eqref{eq:def-ses-Xr}, we get an exact sequence
\begin{equation}\label{eq:mult-omegar-Mde} 
 0\lra H^{n-2}\left(\mc{O}_{X^{(r)}}(e,-d-n)\right) \lra L_{e-r,-d-n-r}\overset{\omega^r}{\lra} L_{e,-d-n} \lra H^{n-1}\left(\mc{O}_{X^{(r)}}(e,-d-n)\right) \lra 0,
\end{equation}
provided $n\geq 3$, or $n=2$ and $e-r\geq -1$, and in particular computing cohomology of line bundles on $X$ (for $r=1$ and $d,e\geq 0$) amounts to understanding the rank of multiplication by $\omega$ on $L$.

\subsection{Monomial complete intersections}
\label{subsec:mon-CI}

We consider the closely related module $M = \bw^n V^{\vee} \oo L$, which differs from $L$ only in its equivariant structure. Relative to the grading induced by the torus action, we get for $\ul{a}=(a_1,\cdots,a_n)\in\bb{Z}^n$ and $\ul{1}=(1,\cdots,1)$, that
\[M_{\ul{a}} = 0\text{ if some $a_i<0$, and otherwise } M_{\ul{a}} = L_{\ul{a}+\ul{1}} = \bigoplus_{d_i+e_i=a_i}\kk\cdot\frac{x_1^{e_1}\cdots x_n^{e_n}}{y_1^{1+d_1}\cdots y_n^{1+d_n}}. \]
We write $T_i = x_iy_i$ and consider the polynomial subalgebra $\kk[T_1,\cdots,T_n]$ of $R$. It acts on $M$ by preserving the $\bb{Z}^n$-graded components, since $\deg(T_i)=\vec{0}$. If $a_i\geq 0$, each $M_{\ul{a}}$ is then a cyclic $\kk[T_1,\cdots,T_n]$-module, given~by
\begin{equation}\label{eq:def-Ma}
M_{\ul{a}} = \kk[T_1,\ldots,T_n]\cdot\frac{1}{y_1^{1+a_1}\cdots y_n^{1+a_n}} \simeq \kk[T_1,\cdots,T_n]/\langle T_1^{1+a_1},\ldots,T_n^{1+a_n}\rangle.
\end{equation}
Multiplication by $\omega$ on $M_{\ul{a}}$ translates via this identification into multiplication by $T=T_1+\cdots+T_n$. If we view $M_{\ul{a}}$ as a standard graded Artinian algebra (where $\deg(T_i)=1$), and if $d+e=a_1+\cdots+a_n$, $e\geq -1$, then 
\begin{equation}\label{eq:H01=Hn-2n-1}
\begin{aligned}
    H^0(\PP,D^d\mc{R}(e))_{\ul{a}} &= H^{n-2}(X,\mc{O}_X(e+1,-d+1-n))_{\ul{a}+\ul{1}} = (0 : T)_e,\quad\text{and}\\
    H^1(\PP,D^d\mc{R}(e))_{\ul{a}} &= H^{n-1}(X,\mc{O}_X(e+1,-d+1-n))_{\ul{a}+\ul{1}} = \left(\frac{M_{\ul{a}}}{T\cdot M_{\ul{a}}}\right)_{e+1}.
\end{aligned}
\end{equation}
More generally, if we restrict \eqref{eq:mult-omegar-Mde} to multidegree $\ul{1}+\ul{a}$, then we get an exact sequence 
\begin{equation}\label{eq:mult-omegar-Ma} 
 0\lra H^{n-2}\left(\mc{O}_{X^{(r)}}(e,-d-n)\right)_{\ul{1}+\ul{a}} \lra (M_{\ul{a}})_{e-r} \overset{T^r}{\lra} (M_{\ul{a}})_e \lra H^{n-1}\left(\mc{O}_{X^{(r)}}(e,-d-n)\right)_{\ul{1}+\ul{a}} \lra 0
\end{equation}
if $n\geq 3$, or $n=2$ and $e-r\geq -1$. We set $m_1=e$, $m_2=-d-n$ and restrict \eqref{eq:les-coh-thick} to multidegree $\ul{a}+\ul{1}$ to~get
\begin{equation}\label{eq:les-Tor01-M/TrM}
\begin{aligned}
   0 &\lra(0:_{M_{\ul{a}}} T)_{e-r} \lra (0:_{M_{\ul{a}}} T^r)_{e-r} \overset{\cdot T}{\lra} (0:_{M_{\ul{a}}} T^{r-1})_{e-r+1} \overset{\partial}{\lra}  \\
   &\lra \left(\frac{M_{\ul{a}}}{TM_{\ul{a}}}\right)_{e-r+1} \overset{T^{r-1}}{\lra} \left(\frac{M_{\ul{a}}}{T^rM_{\ul{a}}}\right)_e \lra \left(\frac{M_{\ul{a}}}{T^{r-1}M_{\ul{a}}}\right)_e \lra 0
\end{aligned}
\end{equation}
Although $e-r\geq -1$ is needed in \eqref{eq:mult-omegar-Ma} when $n=2$, the sequence \eqref{eq:les-Tor01-M/TrM} is always exact (see \eqref{eq:les-Tor-M-kT}). While the infinitesimal thickenings $X^{(r)}$ provide the natural geometric framework, it is more convenient to work in the algebraic setting furnished by \eqref{eq:les-Tor01-M/TrM}. The key is then to understand the connecting map $\partial$. This is achieved through a precise connection with multiplication in the GHM ring, which we develop in Section~\ref{sec:hom-kT}.

\section{The Green--Han--Monsky representation ring}
\label{sec:Han-Monsky}

We recall the discussion of the graded Green--Han--Monsky representation ring $\Delta$ from the Introduction, and proceed to establish some basic results that will be used for our cohomological computations. 

\begin{proposition}\label{prop:basic-mult-GHM}
    If $1\leq a\leq b$ then there exist positive integers $c_j=c_j(a,b)$, $0\leq j<a$, such that
\begin{equation}\label{eq:dela-delb-general} \delta_a\delta_b=\sum_{j=0}^{a-1}\delta_{c_j}(-j).
\end{equation}
Moreover, the integers $c_j$ satisfy the inequalities $c_0 \geq c_1 \geq \cdots \geq c_{a-1}$.
\end{proposition}

\begin{proof}
    We write $M=\kk[T_1,T_2]/(T_1^a,T_2^b)$, $T=T_1+T_2$, and view $M$ as a graded $\kk[T]$-module. It follows for instance by applying \cite{puy-vg}*{Theorem~1.3} to a graded presentation matrix of $M$ that we have a decomposition
    \begin{equation}\label{eq:dirsum-M} 
    M = \bigoplus_{j=0}^r \kk[T]\cdot f_j,
    \end{equation}
    where $f_j$ is homogeneous of degree $d_j$, $d_0\leq d_1\leq\cdots\leq d_r$. The above decomposition implies that the elements $f_j$ form a minimal set of generators of $M$ as a $\kk[T]$-module, hence by Nakayama's lemma they give a basis for the graded vector space $M/TM$. Using
    \[ \begin{aligned}        
    M/TM &= \kk[T_1,T_2]/(T_1^a,T_2^b,T_1+T_2)\simeq \kk[T_1]/(T_1^a,(-T_1)^b)\simeq\kk[T_1]/(T_1^a)\simeq \kk\oplus\kk\cdot T_1\oplus\cdots\oplus\kk\cdot T_1^{a-1} \\   &\simeq\kk\oplus\kk(-1)\oplus\cdots\oplus\kk(-a+1),
    \end{aligned}
    \]
    it follows that $r=a-1$, and $d_j=j$ for $j=0,\cdots,a-1$. Since $T$ acts nilpotently on $M$ we have
    \[\kk[T]\cdot f_j \simeq \kk[T]/(T^{c_j})(-j)\text{ for some }c_j\geq 1,\]
    which implies \eqref{eq:dela-delb-general}. Up to rescaling we may assume that $f_j-T_1^j\in T\cdot M$. It follows that $f_{j+1}-T_1\cdot f_j\in T\cdot M$ and therefore using \eqref{eq:dirsum-M} we can write
    \[f_{j+1}-T_1\cdot f_j = \sum_{s=0}^j \a_s\cdot T^{j+1-s}\cdot f_s,\quad\text{ for some }\a_s\in \kk.\]
    Multiplying the above equality by $T^{c_j}$ and using that $T^{c_j}\cdot f_j=0$, we conclude that
    \[T^{c_j}\cdot f_{j+1} \in \bigoplus_{s=0}^j \kk[T]\cdot f_s. \]
    The decomposition \eqref{eq:dirsum-M} implies then $T^{c_j}\cdot f_{j+1}=0$, hence $c_j\geq c_{j+1}$, concluding our proof.
\end{proof}

We write $\Delta^{u}$ for the \defi{ungraded GHM ring}, and note that Proposition~\ref{prop:basic-mult-GHM} implies that the ungraded multiplication determines the degree shifts in \eqref{eq:dela-delb-general}. If $\op{char}(\kk)=p$ and $a,b\leq q=p^r$, then $c_j\leq q$ for all $j$ (because $T^q=T_1^q+T_2^q$), hence $\delta_1,\cdots,\delta_{q}$ determine a subring $\Delta_q^u$ of $\Delta^u$, isomorphic to the representation ring of a cyclic group of order~$q$ \cite{green}*{Section~2.2}. It follows from \cite{han-monsky}*{Theorem~3.6} or \cite{GPX}*{Theorem~2} that the following subgroups of $\Delta^u$ and $\Delta$ are in fact ideals:
\[ I_p^u = \bigoplus_{d\in\bb{Z}_{\geq 1}} \bb{Z}\cdot\delta_{pd} \subset \Delta^u,\quad I_p = \bigoplus_{j\in\bb{Z},d\in\bb{Z}_{\geq 1}} \bb{Z}\cdot\delta_{pd}(-j)\subset \Delta.\]
We define $\ol{\Delta}^u=\Delta^u/I_p^u$, and $\ol{\Delta}=\Delta/I_p$. With the notation \eqref{eq:dela-delb-general}, we obtain that
\begin{equation}\label{eq:del-prod-div-by-p}
    \text{if $p|ab$ for some $1\leq a\leq b$ then $p|c_j(a,b)$ for all $j=0,\cdots,a-1$.}
\end{equation}
Following \cite{GPX}*{Notation~3}, we also write the multiplication in $\Delta^u$ as
\begin{equation}\label{eq:mult-ungraded-product}
    \delta_a\delta_b = \sum_{i=1}^{\ell} m_i\delta_{\mu_i},\text{ where }\mu_1>\cdots>\mu_{\ell}\geq 1.
\end{equation}
Writing $\mu^m = (\mu,\mu,\cdots,\mu)$, with $\mu$ repeated $m$ times, the relation between \eqref{eq:dela-delb-general} and \eqref{eq:mult-ungraded-product} takes the form
\begin{equation}\label{eq:rel-graded-ungraded}
    a = m_1+\cdots+m_{\ell},\quad\text{and}\quad (c_0,\cdots,c_{a-1}) = (\mu_1^{m_1},\mu_2^{m_2},\cdots,\mu_{\ell}^{m_{\ell}}).
\end{equation}

\begin{lemma}\label{lem:dela-delb-ineq}
    If $p\nmid c$ and $\delta_c(-j)$ is a summand of $\delta_a\delta_b$ in $\Delta$ then 
    \begin{equation}\label{eq:c+2j<a+b} 
     c+2j = a+b-1.
    \end{equation}
\end{lemma}

\begin{proof} It follows from \cite{GPX}*{Theorem~4} that if $\delta_c=\delta_{\mu_i}$ is a summand of $\delta_a\delta_b$ in $\Delta^u$, then $m_i=1$. If $\delta_c(-j)$ is a summand of $\delta_a\delta_b$ in $\Delta$, then it follows from \eqref{eq:dela-delb-general} that $c=c_j$, and from \eqref{eq:mult-ungraded-product} that $j=m_1+\cdots+m_{i-1}$. Using \cite{GPX}*{Theorem~5}, we obtain that
\[ c = \mu_i = a+b-1-2j,\]
which is equivalent to the desired conclusion \eqref{eq:c+2j<a+b}.
\end{proof}

An easy induction now yields the following important generalization.

\begin{proposition}\label{prop:delcj-in-product}
    If $p\nmid c$ and $\delta_c(-j)$ is a summand of $\delta_{a_1}\cdots\delta_{a_n}$ in $\Delta$ then
    \[ c+2j = a_1+\cdots+a_n - (n-1).\]
\end{proposition}

\begin{proof}
    We prove the statement by induction on $n$, noting that $n=2$ is covered in Lemma~\ref{lem:dela-delb-ineq}. Suppose now that $n>2$ and that $\delta_c(-j)$ is a summand of $\delta_{a_1}\cdots\delta_{a_n}$. It follows that $\delta_{a_1}\cdots\delta_{a_{n-1}}$ contains a summand $\delta_b(-i)$ such that $\delta_c(-(j-i))$ is a summand of $\delta_b\cdot\delta_{a_n}$. By \eqref{eq:del-prod-div-by-p}, the hypothesis $p\nmid c$ implies that $p\nmid b$. We can therefore apply the induction hypothesis to conclude that
    \[ b+2i = a_1+\cdots+a_{n-1}-(n-2).\]
    Applying Lemma~\ref{lem:dela-delb-ineq} to the product $\delta_b\cdot\delta_{a_n}$, we conclude that
    \[ c+2(j-i) = b+a_n-1.\]
    Adding the two displayed equalities above and canceling $b$ concludes our proof.
\end{proof}

While Proposition~\ref{prop:delcj-in-product} uniquely identifies the degree shift of a summand $\delta_c$, it is not a priori clear for which values of $c$ with $p\nmid c$ there is a summand $\delta_c$ of $\delta_{a_1}\cdots\delta_{a_n}$. We show next that in characteristic $p=2$ this summand is uniquely determined. As in \eqref{eq:def-nim-sum}, we write $a\oplus b$ for the \defi{Nim-sum} of $a$ and $b$.

\begin{proposition}\label{prop:delab-odd-Nim}
    Suppose that $\op{char}(\kk)=2$. For $a,b,c\geq 0$, we have that $\delta_{2c+1}(-j)$ appears as a summand of $\delta_{2a+1}\delta_{2b+1}$ in $\Delta$ if and only if
    \[ c = a\oplus b \quad\text{ and }\quad j = a+b-c.\]
\end{proposition}

\begin{proof} We know by Lemma~\ref{lem:dela-delb-ineq} that if $\delta_{2c+1}(-j)$ is a summand of $\delta_{2a+1}\delta_{2b+1}$ in $\Delta$ then $j=a+b-c$. Therefore it suffices to consider the ungraded multiplication, and prove that $c=a\oplus b$ gives the unique summand $\delta_{2c+1}$ of $\delta_{2a+1}\delta_{2b+1}$ in $\Delta^u$. We do so by induction on $a,b$, noting that the base case $a=b=0$ is trivially satisfied. 

We assume $a\leq b$ and for some $n\geq 1$, write 
\[a=2^{n-1}a_0+a_1,\quad b=2^{n-1}+b_1,\] 
where $a_0\in\{0,1\}$, and $0\leq a_1,b_1<2^{n-1}$. It follows by induction that $\delta_{2a_1+1}\delta_{2b_1+1}$ in $\Delta^u$ has a unique summand of the form $\delta_{2c_1+1}$, which occurs for $c_1=a_1\oplus b_1$. Moreover, \cite{GPX}*{Proposition~6}, \cite{renaud}*{Theorem~2}, imply that there is a unique odd summand of $\delta_{2a+1}\delta_{2b+1}$ in $\Delta^u$, namely $\delta_{2c+1}$ where $c = (1-a_0)2^{n-1}+c_1$. Since $1-a_0=a_0\oplus 1$ and $c_1=a_1\oplus b_1$, it follows that $c=a\oplus b$, as desired.
\end{proof}

Note that in light of \eqref{eq:dela-delb-general}, the multiplicity of $\delta_{2c+1}(-j)$ in $\delta_{2a+1}\delta_{2b+1}$ is exactly one if it is non-zero. By induction, this yields the following important consequence.

\begin{corollary}\label{cor:prod-dela-odd-Nim}
    Suppose that $\op{char}(\kk)=2$. If $a_1,\cdots,a_n,c\geq 0$, we have that $\delta_{2c+1}(-j)$ appears as a summand of $\delta_{2a_1+1}\delta_{2a_2+1}\cdots\delta_{2a_n+1}$ if and only if
    \[ c = a_1\oplus a_2\oplus\cdots\oplus a_n \quad\text{ and }\quad j = a_1+\cdots+a_n-c.\]
    Moreover, if $\delta_{2c+1}(-j)$ appears as a summand, then its multiplicity is equal to one.
\end{corollary}

\section{Some homological algebra over $\kk[T]$}
\label{sec:hom-kT}

Our interest in the GHM ring arises in relation to some basic homological algebra calculations over $\kk[T]$, which foreshadow our later study of line bundle cohomology on infinitesimal thickenings of the incidence correspondence. We consider for each $r\geq 1$ the short exact sequence
\begin{equation}\label{eq:ses-kT-mod-Tr} 0 \lra \frac{(T^{r-1})}{(T^r)}\lra \frac{\kk[T]}{(T^r)} \lra \frac{\kk[T]}{(T^{r-1})} \lra 0
\end{equation}
and note that we have an isomorphism $\frac{(T^{r-1})}{(T^r)}\simeq\frac{\kk[T]}{T}(-r+1)=\kk(-r+1)$ of graded $\kk[T]$-modules. Suppose now that $M$ is a graded $\kk[T]$-module, and tensor the exact sequence \eqref{eq:ses-kT-mod-Tr} with $M$ to get a long exact sequence
\[
\begin{aligned}
    0 &\lra \Tor_1^{\kk[T]}(M,\kk(-r+1)) \lra \Tor_1^{\kk[T]}\left(M,\frac{\kk[T]}{(T^r)}\right) \lra \Tor_1^{\kk[T]}\left(M,\frac{\kk[T]}{(T^{r-1})}\right) \overset{\partial_M}{\lra} \\
    & \lra \frac{M}{TM}(-r+1) \lra \frac{M}{T^rM} \lra \frac{M}{T^{r-1}M} \lra 0
\end{aligned}
\]
If we write
\[ (0:_M f) = \{m\in M : fm=0\} \quad\text{for }f\in\kk[T],\]
then there exists an identification
\[\Tor_1^{\kk[T]}\left(M,\frac{\kk[T]}{(T^r)}\right) = (0:_M T^r)(-r),\]
hence the above long exact sequence becomes
\begin{equation}\label{eq:les-Tor-M-kT}
0 \to (0:_M T)(-r) \to (0:_M T^r)(-r) \overset{\cdot T}{\to} (0:_M T^{r-1})(-r+1) \overset{\partial_M^r}{\to} \frac{M}{TM}(-r+1) \overset{T^{r-1}}{\to} \frac{M}{T^rM} \to \frac{M}{T^{r-1}M} \to 0
\end{equation}
We are interested in the rank of the connecting homomorphism $\partial_M^r$ in each degree, in the case where $M$ is a finite length graded module. We will abuse notation and write 
\[ M = \sum_{c,j} m_{c,j}\delta_c(-j)\quad\text{if}\quad M = \bigoplus_{c,j} \left(\frac{\kk[T]}{(T^c)}(-j)\right)^{\oplus m_{c,j}}.\]

\begin{lemma}\label{lem:rank-partialM}
If $M = \sum_{c,j} m_{c,j}\delta_c(-j)$ then 
\[ \rk(\partial_{M,t}^r) = \sum_{c=1}^{r-1} m_{c,t-r+1}.\]
\end{lemma}

\begin{proof}
    Since the exact sequence \eqref{eq:les-Tor-M-kT} is compatible with direct sum decompositions, it suffices to analyze the case when $M=\delta_c(-j)$. Note that by the exactness of \eqref{eq:les-Tor-M-kT}, we have
    \[ \rk(\partial_{M,t}^r) = \dim_{\kk}\ker\left(\left(\frac{M}{TM}\right)_{t-r+1} \overset{T^{r-1}}{\to} \left(\frac{M}{T^rM}\right)_t \right)\]
    Since $M/TM=\kk(-j)$, we have that $(M/TM)_{t-r+1}\neq 0$ if and only if $j=t-r+1$, in which case $\dim_{\kk}(M/TM)_{t-r+1}=1$. Moreover, the multiplication by $T^{r-1}$ sends $(M/TM)_{t-r+1}$ to zero if and only if $c\leq r-1$. It follows that $\rk(\partial^r_{M,t}) = 1$ when $j=t-r+1$ and $c\leq r-1$, and $\rk(\partial^r_{M,t}) = 0$ otherwise, concluding our proof.
\end{proof}

We now apply these observations to the monomial complete intersection $M_{\ul{a}}$ in \eqref{eq:def-Ma}, which is viewed as a $\kk[T]$-module for $T=T_1+\cdots+T_n$. We consider the Hilbert series of $M_{\ul{a}}/TM_{\ul{a}}$,
\begin{equation}\label{eq:Pat}
    P_{\ul{a}}(t) = \sum_{j\geq 0} \dim_{\kk}\left(\frac{M_{\ul{a}}}{TM_{\ul{a}}} \right)_j \cdot t^j.
\end{equation}
We write $M_{\ul{a}} = \delta_{1+a_1}\cdots\delta_{1+a_n}= \sum_{c,j} m_{c,j}(\ul{a}) \delta_c(-j)$, and for each $c$ we consider the polynomial
\begin{equation}\label{eq:def-fcat}
f_c^{\ul{a}}(t) = \sum_{j\geq 0}m_{c,j}(\ul{a}) \cdot t^j
\end{equation}
that keeps track of the indecomposable $\kk[T]$-summands of $M_{\ul{a}}$ of length $c$. While $f^{\ul{a}}_c(t)$ can be complicated in general, we will mainly need the range of values $c=1,\cdots,p-1$. In this case, since $c$ not divisible by $p$, 
Proposition~\ref{prop:delcj-in-product} implies that whenever $f_c^{\ul{a}}(t)$ is non-zero, it consists of a single term:
\begin{equation}\label{eq:fcat-small-c}
f_c^{\ul{a}}(t) = m_{c,j}(\ul{a}) \cdot t^j \quad\text{where}\quad j=\frac{a_1+\cdots+a_n+1-c}{2}.
\end{equation}

\begin{theorem}\label{thm:Pa-recursion}
    With the convention $P_{\ul{b}}(t)=0$ if $b_s<0$ for some $1\leq s\leq n$, we have 
    \[P_{\ul{a}}(t)\cdot (1+t+\cdots+t^{p-1}) = \sum_{0\leq i_1,\cdots,i_n\leq p-1} P_{\ul{a}'}(t^p)\cdot t^{i_1+\cdots+i_n} + \sum_{c=1}^{p-1}f^{\ul{a}}_c(t) \cdot (t^c+t^{c+1}+\cdots+t^{p-1}),\]
    where in the first sum $\ul{a}'=\ul{a}'(\ul{i})$ is defined by
    \begin{equation}\label{eq:def-aprs} 
    a'_s = \left\lfloor\frac{a_s-i_s}{p}\right\rfloor\quad\text{for }s=1,\cdots,n.
    \end{equation}
\end{theorem}

\begin{proof} We extend the definition \eqref{eq:Pat} by setting
\[    P^r_{\ul{a}}(t) = \sum_{j\geq 0} \dim_{\kk}\left(\frac{M_{\ul{a}}}{T^rM_{\ul{a}}} \right)_j \cdot t^j\quad\text{for }r\geq 0.\]
It follows from Lemma~\ref{lem:rank-partialM} by taking Euler characteristics that
\[t^{r-1}\cdot P_{\ul{a}}(t) - P^{r}_{\ul{a}}(t) + P^{r-1}_{\ul{a}}(t) = \sum_{c=1}^{r-1} f^{\ul{a}}_c(t)\cdot t^{r-1},\]
where the right side encodes the Hilbert series of the image of the connecting homomorphism $\partial_{M_{\ul{a}}}^r$. Summing the above identities for $r=1,2,\cdots,p$, and using the fact that $P^0_{\ul{a}}(t)=0$, we get
\[P_{\ul{a}}(t)\cdot (1+t+\cdots+t^{p-1}) - P^p_{\ul{a}}(t) = \sum_{r=1}^p \sum_{c=1}^{r-1}f^{\ul{a}}_c(t) \cdot t^{r-1} = \sum_{c=1}^{p-1} f^{\ul{a}}_c(t)\cdot (t^c+t^{c+1}+\cdots+t^{p-1}),\]
where the second equality follows by reversing the summation order. In order to prove our result, it remains to verify that
\[P^p_{\ul{a}}(t) = \sum_{0\leq i_1,\cdots,i_n\leq p-1} P_{\ul{a}'}(t^p)\cdot t^{i_1+\cdots+i_n}.\]
To that end we note that $T^p=T_1^p+\cdots+T_n^p$, and define the Frobenius twist of $M_{\ul{a}}$ to be
\[M_{\ul{a}}^{(p)} = \kk[T_1^p,\cdots,T_n^p]/\langle T_1^{p(1+a_1)}\cdots T_n^{p(1+a_n)}\rangle,
\]
which is a $\kk[T^p]$-module. Moreover, we have a direct sum decomposition of $M_{\ul{a}}$ as a $\kk[T^p]$-module, given by
\[M_{\ul{a}} = \bigoplus_{0\leq i_1,\cdots,i_n\leq p-1}M_{\ul{a}'}^{(p)} \cdot T_1^{i_1}\cdots T_n^{i_n},\]
where $\ul{a}'=\ul{a}'(\ul{i})$ is defined as in \eqref{eq:def-aprs}. We have an identification
\[\frac{M_{\ul{a}'}^{(p)}}{T^pM_{\ul{a}'}^{(p)}} = \frac{M_{\ul{a}'}}{TM_{\ul{a}'}} \oo_{\kk[T]} \kk[T^p],\]
where in the tensor product the map $\kk[T]\lra\kk[T^p]$ is the isomorphism sending $T\lra T^p$. It follows that the Hilbert series of $M_{\ul{a}'}^{(p)}/T^pM_{\ul{a}'}^{(p)}$ is computed by $P_{\ul{a}'}(t^p)$, from which the desired conclusion follows.
\end{proof}

\section{Verlinde algebras and the reduced GHM ring}
\label{sec:Verlinde}

To better understand the polynomials in \eqref{eq:fcat-small-c}, we explain following \cite{CEO} the connection between the reduced GHM ring $\ol{\Delta}^u$ and the Verlinde algebras for $\op{SU}(2)$. This connection leads, via the Verlinde formula, to explicit formulas for the relevant structure constants.

\subsection{Verlinde algebras and their (higher) structure constants}
\label{subsec:Ver-structure-constants}

Following \cite{verlinde}*{Section~4}, we define the \defi{Verlinde algebra $\mc{V}_k$} of $\op{SU}(2)$ at level $k$ to be the commutative $\bb{Z}$-algebra with basis $\phi_0,\cdots,\phi_k$ and multiplication given, for $0\leq a\leq b\leq k$, by
\begin{equation}\label{eq:def-phia-phib}
\phi_a\cdot\phi_b = \sum_{j=0}^{\min(a,k-b)} \phi_{b-a+2j} = \phi_{b-a}+\phi_{b-a+2}+\cdots+\phi_{\min(a+b,2k-a-b)}.
\end{equation}
We will be interested in the \defi{higher structure constants $V_k(a_1,\cdots,a_n;c)$} for multiplication in $\mc{V}_k$, defined by
\[ \phi_{a_1}\cdots\phi_{a_n} = \sum_{c=0}^k V_k(a_1,\cdots,a_n;c) \cdot \phi_c.\]
When $n=2$, we get from \eqref{eq:def-phia-phib} that
\[V_k(a,b;c) = \begin{cases}
    1 & \text{if }|a-b|\leq c\leq\min(a+b,2k-a-b),\ c\equiv a+b\ (\op{mod }2),\\
    0 & \text{otherwise.}
\end{cases}\]
For general $n$, we encode these structure coefficients by defining the \defi{Verlinde symmetric polynomials}
\begin{equation}\label{eq:versym-kc-polys}
\vartheta^{k,c}(z_1,\cdots,z_n) = \sum_{0\leq a_1,\cdots,a_n\leq k} V_k(a_1,\cdots,a_n;c) \cdot z_1^{a_1}\cdots z_n^{a_n}.
\end{equation}
We note that $\vartheta^{k,c}(\ul{z})$ is usually not homogeneous, and we write $\vartheta^{k,c}_d(\ul{z})$ for the degree $d$ component, which can only be nonzero if $d-c$ is even.

\begin{example}\label{ex:Verlinde-polys-k=1}
We consider the special case when $k=1$, when multiplication is given by $\phi_0^2=\phi_1^2=\phi_0$ and $\phi_0\phi_1=\phi_1$. It follows that
\[
V_k(a_1,\cdots,a_n;c) = \begin{cases}
    1 & \text{if }a_1+\cdots+a_n\equiv c\ (\op{mod }2),\\
    0 & \text{otherwise}.
\end{cases}
\]
In particular each non-zero $\vartheta^{1,c}_d(\ul{z})$ is an elementary symmetric polynomial, and we can write
\[\vartheta^{1,0}(\ul{z}) =\frac{1}{2}\cdot\left(\prod_{i=1}^n (1+z_i)+\prod_{i=1}^n (1-z_i) \right),\quad \vartheta^{1,1}(\ul{z}) = \frac{1}{2}\cdot\left(\prod_{i=1}^n (1+z_i)-\prod_{i=1}^n (1-z_i) \right).\]
\end{example}

To obtain general formulas beyond $k=1$, we will use the Verlinde formula: there exists an orthogonal matrix $(S_{l,r})$ given by \cite{verlinde}*{(4.10)}
\begin{equation}\label{eq:def-Smatrix} 
S_{l,r} = \sqrt{\frac{2}{k+2}}\sin\left(\frac{(l+1)(r+1)\pi}{k+2} \right),\quad l,r=0,\cdots,k,
\end{equation}
and the characters of the algebra $\mc{V}_k$ are real, defined by \cite{verlinde}*{(3.12)}
\begin{equation}\label{eq:characters-Vk}
    \psi_r:\mc{V}_k \lra \bb{R},\quad \psi_r(\phi_l) = \frac{S_{l,r}}{S_{0,r}}\quad\text{ for }l,r=0,\cdots,k.
\end{equation}
We can use \eqref{eq:characters-Vk} and the orthogonality relations to determine the coefficients in any expression $\phi = \sum_{l=0}^k \a_l\phi_l$, as follows. We have
\begin{equation}\label{eq:ac-from-Slr-psir}
\a_c = \sum_{l=0}^k \a_l\left(\sum_{r=0}^k S_{l,r} S_{c,r}\right) = \sum_{r=0}^k S_{0,r} S_{c,r} \left(\sum_{l=0}^k \a_l\frac{S_{l,r}}{S_{0,r}}\right) = \sum_{r=0}^k S_{0,r}S_{c,r}\psi_r(\phi).
\end{equation}
When $\phi=\phi_a\phi_b = \sum_{l=0}^k V_k(a,b;l)\phi_l$, this is Verlinde's formula \cite{verlinde}*{(3.11)} for the structure coefficients:
\[V_k(a,b;c) = \sum_{r=0}^k\frac{S_{a,r}S_{b,r}S_{c,r}}{S_{0,r}}.\]
Similar reasoning provides the following description of the Verlinde symmetric polynomials \eqref{eq:versym-kc-polys}.

\begin{proposition}\label{prop:trig-Verlinde-polys}
   If we write $\theta_r = \frac{r\pi}{k+2}$, then for $0\leq c\leq k$ we have
\[\vartheta^{k,c}(\ul{z}) = \frac{2}{k+2}\cdot\sum_{r=1}^{k+1}\sin\left(\theta_r\right)\sin\left((c+1)\theta_r\right)\cdot\prod_{i=1}^n\frac{1-(-1)^r\cdot z_i^{k+2}}{1-2\cos\left(\theta_r\right)z_i+z_i^2}.\]
\end{proposition}

\begin{proof} We define 
\[\phi(z) = \sum_{c=0}^k \phi_c z^c,\]
and expand the product
\[\begin{aligned}
    \prod_{i=1}^n \phi(z_i) &= \sum_{a_1,\cdots,a_n=0}^k \phi_{a_1}\cdots\phi_{a_n} z_1^{a_1}\cdots z_n^{a_n} \\
    &= \sum_{a_1,\cdots,a_n,c=0}^k V_k(a_1,\cdots,a_n;c) z_1^{a_1}\cdots z_n^{a_n} \phi_c \overset{\eqref{eq:versym-kc-polys}}{=} \sum_{l=0}^k \vartheta^{k,l}(z_1,\cdots,z_n)\cdot \phi_l.
\end{aligned}\]
Extending the definition \eqref{eq:characters-Vk} by letting $\psi_r(z_i)=z_i$, it follows from \eqref{eq:ac-from-Slr-psir} that
\begin{equation}\label{eq:varthetakc-fourier-formula}
\vartheta^{k,c}(z_1,\cdots,z_n) = \sum_{r=0}^k S_{0,r}S_{c,r}\prod_{i=1}^n\psi_r\left(\phi(z_i)\right).
\end{equation}
If we make the convention $\phi_{-1}=\phi_{k+1}=0$, then \eqref{eq:def-phia-phib} implies $\phi_1\phi_{c-1}=\phi_c+\phi_{c-2}$ for $c=1,\cdots,k+1$, hence
\[(1-\phi_1 z + z^2)\cdot \phi(z) = 1 + \left(\sum_{c=1}^{k+1}(\phi_c-\phi_{c-1}\phi_1+\phi_{c-2})z^c\right) + \phi_k z^{k+2} = 1 + \phi_k z^{k+2}.\]
Applying $\psi_r$ and noting that 
\[\psi_r(\phi_1) = \frac{\sin(2\theta_{r+1})}{\sin(\theta_{r+1})} = 2\cos(\theta_{r+1}),\quad \psi_r(\phi_k) = \frac{\sin((r+1)\pi-\theta_{r+1})}{\sin(\theta_{r+1})}=(-1)^r,\]
it follows that 
\[\psi_r(\phi(z)) = \frac{1+(-1)^r z^{k+2}}{1-2\cos(\theta_{r+1})z+z^2}.\]
The desired conclusion now follows from \eqref{eq:varthetakc-fourier-formula}, using the identity $S_{0,r}S_{c,r}=\frac{2}{k+2}\sin(\theta_{r+1})\sin((c+1)\theta_{r+1})$, and reindexing so that $r$ runs from $1$ to $k+1$.
\end{proof}

We will also be interested in the algebra $\mc{V}'_k$, defined as the \defi{even part} of $\mc{V}_{2k}$: it is a commutative $\bb{Z}$-algebra with basis $\phi'_a = \phi_{2a}$, $a=0,\cdots,k$, and multiplication given for $0\leq a\leq b\leq k$ by
\begin{equation}\label{eq:def-phi'a-phi'b}
\phi'_a\cdot\phi'_b = \sum_{j=0}^{\min(2a,2k-2b)} \phi'_{b-a+j} = \phi'_{b-a}+\phi'_{b-a+1}+\cdots+\phi'_{\min(a+b,2k-a-b)}.
\end{equation}
In analogy with \eqref{eq:versym-kc-polys}, we consider the \defi{even Verlinde symmetric polynomials} 
\begin{equation}\label{eq:even-versym-kc-polys}
w^{k,c}(z_1,\cdots,z_n) = \sum_{0\leq a_1,\cdots,a_n\leq k} V_{2k}(2a_1,\cdots,2a_n;2c) \cdot z_1^{a_1}\cdots z_n^{a_n},
\end{equation}
for which we will derive a formula analogous to the one in Proposition~\ref{prop:trig-Verlinde-polys}. The characters of $\mc{V}'_k$ are obtained by restriction from $\mc{V}_{2k}$, and are given by
\[\psi'_r:\mc{V}'_k \lra \bb{R},\quad \psi'_r(\phi'_l) = \psi_r(\phi_{2l}),\quad\text{ for }l,r=0,\cdots,k.\]
The obvious restriction of the $S$-matrix of $\mc{V}_{2k}$ is no longer orthogonal, so we consider instead the orthogonal matrix $(S'_{l,r})$ defined by
\begin{equation}\label{eq:def-S'matrix} 
S'_{l,r} = \begin{cases}
    \sqrt{\frac{2}{k+1}}\sin\left(\frac{(2l+1)(r+1)\pi}{2k+2} \right), & \text{if }0\leq r\leq k-1,\ 0\leq l\leq k, \\
    \frac{(-1)^l}{\sqrt{k+1}}, & \text{if } r=k,\ 0\leq l\leq k, 
\end{cases}
\end{equation}
and note that
\[\psi'_r(\phi'_l) = \frac{S'_{l,r}}{S'_{0,r}} = \frac{\sin\left(\frac{(2l+1)(r+1)\pi}{2k+2} \right)}{\sin\left(\frac{(r+1)\pi}{2k+2} \right)}\quad\text{for }l,r=0,\cdots,k.\]

\begin{proposition}\label{prop:trig-even-Verlinde-polys}
   If we write $\theta'_r = \frac{r\pi}{2k+2}$, and $\delta_{l,r}$ for the Kronecker delta, then for $0\leq c\leq k$ we have
\[w^{k,c}(\ul{z}) = \sum_{r=1}^{k+1}\frac{2-\delta_{r,k+1}}{k+1}\cdot\sin(\theta'_{r})\cdot\sin((2c+1)\theta'_{r})\cdot\prod_{i=1}^n\frac{(1+z_i)(1-(-1)^r\cdot z_i^{k+1})}{1-2\cos(2\theta'_r)z_i+z_i^2}.\]
\end{proposition}

\begin{proof}  We define 
\[\phi'(z) = \sum_{c=0}^k \phi'_c z^c,\]
and expand the product
\[
    \prod_{i=1}^n \phi'(z_i) = \sum_{l=0}^k w^{k,l}(z_1,\cdots,z_n)\cdot \phi'_l.
\]
In analogy with \eqref{eq:varthetakc-fourier-formula}, we obtain
\begin{equation}\label{eq:wkc-fourier-formula}
w^{k,c}(z_1,\cdots,z_n) = \sum_{r=0}^k S'_{0,r}S'_{c,r}\prod_{i=1}^n\psi'_r\left(\phi'(z_i)\right).
\end{equation}
It follows from \eqref{eq:def-phi'a-phi'b} that $(\phi'_1-1)\phi'_{c-1}=\phi'_c+\phi'_{c-2}$ for $c=2,\cdots,k$, and $\phi'_1\phi'_k=\phi'_{k-1}$, hence
\[(1-(\phi'_1-1) z + z^2)\cdot \phi'(z) = 1 + z +  \left(\sum_{c=2}^{k}(\phi'_c-\phi'_{c-1}(\phi'_1-1)+\phi'_{c-2})z^c\right) + \phi'_k(z^{k+1}+z^{k+2}) = (1+z)(1 + \phi'_k z^{k+1}).\]
Applying $\psi'_r$ and noting that 
\[\psi'_r(\phi'_1-1) = \frac{\sin(3\theta'_{r+1})-\sin(\theta'_{r+1})}{\sin(\theta'_{r+1})} = 2\cos(2\theta'_{r+1}),\quad \psi'_r(\phi'_k) = \frac{\sin((r+1)\pi-\theta'_{r+1})}{\sin(\theta'_{r+1})}=(-1)^r,\]
it follows that 
\[\psi'_r(\phi'(z)) = \frac{(1+z)(1+(-1)^r z^{k+1})}{1-2\cos(2\theta'_{r+1})z+z^2}.\]
The formula \eqref{eq:def-S'matrix} is equivalent to
\[ S'_{l,r} = \sqrt{\frac{2-\delta_{r,k}}{k+1}}\sin\left((2l+1)\theta'_{r+1}\right),\]
which applied to \eqref{eq:wkc-fourier-formula} yields (after reindexing) the desired formula for $w^{k,c}(\ul{z})$ and concludes our proof.
\end{proof}

We will be interested in the Verlinde symmetric polynomials $\vartheta^{k,c}(\ul{z})$ for $k=p-2$, and $w^{k,c}(\ul{z})$ for $k=p-1$. In that case we have $\theta_r = r\pi/p=2\theta'_r$. Comparing \eqref{eq:def-Acz} with Proposition~\ref{prop:trig-Verlinde-polys}, we obtain
\begin{equation}\label{eq:Ac=vthetap-2c-1} A_c(\ul{z}) = \vartheta^{p-2,c-1}(\ul{z})\quad\text{for }c=1,\cdots,p-1.
\end{equation}
Similarly, \eqref{eq:def-Thetaz-Theta'z} together with Proposition~\ref{prop:trig-even-Verlinde-polys} yields
\begin{equation}\label{eq:Thetaz=wp-1} 
\Theta(\ul{z}) = w^{p-1,0}(\ul{z}),\quad \Theta'(\ul{z}) = w^{p-1,p-1}(\ul{z}).
\end{equation}

\subsection{Verlinde polynomials and modular characters}
\label{subsec:Verlinde-from-tableaux}

Our next goal is to give a more combinatorial interpretation of the Verlinde polynomials, and explain how the graded components of $A_c(\ul{z})$ are simple modular characters of $\GL_n$. To that end we interpret the multiplication \eqref{eq:def-phia-phib} as a \defi{truncated Pieri rule} as follows.

Recall that to a partition $\ll=(\ll_1,\cdots,\ll_\ell)$ we can associate a \defi{Young diagram} of shape $\ll$, which is a left-justified array of boxes with $\lambda_i$ boxes in row $i$. A \defi{Young tableau} $T$ of shape $\lambda$ is a filling of these boxes with positive integers, and the weight $\ul{a}=(a_1,a_2,\cdots)$ of $T$ is defined by letting $a_i$ count the number of entries of $T$ equal to $i$. We say that $T$ is \defi{semistandard} if entries weakly increase along rows and strictly increase down columns. For a semistandard tableau $T$, we write $T[i]$ for the subtableau whose entries are $\leq i$. For $\ll=(5,2)$ we display below the corresponding Young diagram, along with a semistandard tableau $T$ of shape $\ll$, and some of its subtableaux of the form $T[i]$. 
\[
\ytableausetup{aligntableaux=center}
\ydiagram{5,2}\qquad
T=\begin{ytableau}
1 & 1 & 2 & 3 & 3\\
2 & 2 
\end{ytableau}\qquad
T[2]=\begin{ytableau}
1 & 1 & 2\\
2 & 2 
\end{ytableau}\qquad
T[1]=\begin{ytableau}
1 & 1  
\end{ytableau}
\]
Given a semistandard tableau $T$ with $2$ rows, we say that $T$ is \defi{$k$-restricted} if it satisfies the following condition for each $i$: if we write $\mu=(\mu_1,\mu_2)$ for the shape of $T[i]$, then the entries of $T[i+1]\setminus T[i]$ are contained in columns $\mu_2+1,\cdots,\mu_2+k$ (in particular, the weight $\ul{a}$ of $T$ satisfies $a_j\leq k$ for all $j$). The tableau $T$ above is $3$-restricted, but not $2$-restricted, while the displayed tableaux below are not $3$-restricted:
\[
\ytableausetup{aligntableaux=center}
\begin{ytableau}
1 & 1 & 2 & 2 & 3\\
2 & 3 
\end{ytableau}\qquad\qquad\qquad
\begin{ytableau}
1 & 1 & 2 & 2 & 2\\
3 & 3 
\end{ytableau}
\]
If we associate to a tableau of shape $\ll=(\ll_1,\ll_2)$ the element $\phi_{\ll_1-\ll_2}$ in $\mc{V}_k$, then we get a bijection between the summands in \eqref{eq:def-phia-phib} and the $k$-restricted tableaux of weight $(a,b)$: indeed, any such tableau $T$ will have $a$ entries equal to $1$ and $i$ entries equal to $2$ in the first row, and $b-i$ more entries in the second row; additionally, it will have to satisfy $b-i\leq a$ in order to be semistandard, and $a+i\leq k$ in order to be $k$-restricted, so $\max(0,b-a)\leq i\leq k-a$. The corresponding element in $\mc{V}_k$ is $\phi_{a-b+2i}$, which yields the desired identification.

\begin{example}\label{ex:Pieri-k=5} For a concrete example of the truncated Pieri rule described above we take $k=5$ and consider the products $\phi_3\cdot\phi_4=\phi_4\cdot\phi_3$. Identifying a tableau with the corresponding basis element in $\mc{V}_k$, we have
\[
\ytableausetup{aligntableaux=center}
\begin{aligned}
 \phi_3\cdot\phi_4 &= \ytableaushort{1112,222} + \ytableaushort{11122,22} = \phi_1+\phi_3\\
 \phi_4\cdot\phi_3 &= \ytableaushort{1111,222} + \ytableaushort{11112,22} = \phi_1+\phi_3
\end{aligned}
\]    
\end{example}
Arguing by induction, it follows that if $a_1+\cdots+a_n=2m+c$ then
\begin{equation}\label{eq:Vka1n}
    V_k(a_1,\cdots,a_n;c) = \text{the number of $k$-restricted tableaux of shape $(m+c,m)$ and weight $\ul{a}$.}
\end{equation}
Going back to the case $k=3$ and $\ul{a}=(2,3,2)$, we get $\phi_2\cdot\phi_3\cdot\phi_2 = \phi_1+\phi_3$, with corresponding tableaux
\[\ytableausetup{aligntableaux=center}
\begin{ytableau}
1 & 1 & 2 & 3\\
2 & 2 & 3
\end{ytableau}\qquad\text{and}\qquad
\begin{ytableau}
1 & 1 & 2 & 3 & 3\\
2 & 2 
\end{ytableau}\]

We let $\mc{T}^{k,2}_{res}$ denote the set of $k$-restricted tableaux (which are semistandard and have at most $2$ rows). We write $T_{u,v}$ for the entry in row $u$, column $v$ of $T$, and prove the following property of tableaux in~$\mc{T}^{k,2}_{res}$.

\begin{lemma}\label{lem:kstrict=ineqTuv}
    If $T\in\mc{T}^{k,2}_{res}$ has shape $\ll=(\ll_1,\ll_2)$ then 
    \begin{equation}\label{eq:ineqTuv-kstrict} 
    T_{(2,v)} < T_{(1,v+k)}\text{ for all }1\leq v\leq \min(\ll_2,\ll_1-k).
    \end{equation}
\end{lemma}

\begin{proof}
    Suppose that $T\in\mc{T}^{k,2}_{res}$, consider $1\leq v\leq \min(\ll_2,\ll_1-k)$, and write $T_{(2,v)}=i+1$. We let $\mu=(\mu_1,\mu_2)$ be the shape of $T[i]$, so that $\mu_2<v$ and $\mu_1\leq\mu_2+k<v+k$. By definition, the entries of $T[i+1]\setminus T[i]$ are contained in columns $\mu_2+1,\cdots,\mu_2+k<v+k$, hence $T[i+1]$ has at most $v+k-1$ columns. This implies that $T_{(1,v+k)}>i+1=T_{(2,v)}$, as desired.
\end{proof}

To build towards the framework of \cite{mat-pap}, we say that a partition (or Young diagram) $\ll$ is \defi{$(k,r)$-special} if $\ll_1\leq k$ and $\ll$ has at most $r$ parts of size $<k$: equivalently, we can write $\ll=(\ll_1,\cdots,\ll_{s+t})$, where $\ll_1=\cdots=\ll_s=k$, and $t\leq r$. We say that a semistandard tableau $T$ is \defi{$(k,r)$-semistandard} if the shape of $T[i]$ is $(k,r)$-special for all $i$, and write $\mc{T}^{k,r}_{ss}$ for the set of $(k,r)$-semistandard tableaux. With this notation, we have that for $m<p$ the partition $\ll$ is $(m,p-m)$-special if and only if it is $m$-special in the sense of \cite{mat-pap}; similarly, a tableau $T$ is $(m,p-m)$-semistandard if and only if it is $m$-semistandard in the sense of  \cite{mat-pap}. 

\begin{lemma}\label{lem:special-vs-restricted}
    There exists a weight-preserving bijection between $\mc{T}^{k,2}_{res}$ and $\mc{T}^{k,2}_{ss}$.
\end{lemma}

\begin{proof}
    We write $T$ for a typical tableau in $\mc{T}^{k,2}_{res}$ and $\ll=(\ll_1,\ll_2)$ for its shape, noting that $\ll_1\leq \ll_2+k$. Similarly we write $S$ for a tableau in $\mc{T}^{k,2}_{ss}$, and $\mu$ for its shape, which can be written as
\begin{equation}\label{eq:k2special-mu}
    \mu=(\underbrace{k,\ldots,k}_{s\text{ times}},a,b),
\qquad \text{where }k>a\geq b\geq 0.
\end{equation}
Our goal is to define weight-preserving inverse transformations
\[ \a:\mc{T}^{k,2}_{res}\lra\mc{T}^{k,2}_{ss},\quad \b:\mc{T}^{k,2}_{ss}\lra \mc{T}^{k,2}_{res}.\]
Given $T\in \mc{T}^{k,2}_{res}$, we write $A$, $B$ for its row words, and divide each row into consecutive blocks of size $k$, allowing the last block to be shorter:
\[
A=A^{(1)}A^{(2)}\cdots,
\qquad
B=B^{(1)}B^{(2)}\cdots .
\]
We let $S$ be the tableau with rows $A^{(1)},B^{(1)},A^{(2)},B^{(2)},\cdots$, and set $\a(T)=S$. The fact that $S$ is semistandard follows from Lemma~\ref{lem:kstrict=ineqTuv}, so to show that $\a$ is well-defined we have to verify that each $S[i]$ has a $(k,2)$-special shape. Note that $T[i]\in \mc{T}^{k,2}_{res}$ and by construction $S[i]=\a(T)[i]=\a(T[i])$, so we may without loss of generality assume that $S=S[i]$. We write $\ll_2=tk+r$ with $0\leq r<k$, and note that $\ll_1\leq \ll_2+k$ is given by either
\[ (1)\ \ll_1 = tk+r',\text{ where }r\leq r'<k,\quad\text{or}\quad (2)\ \ll_1=(t+1)k+r',\text{ where }0\leq r'\leq r.\]
It follows that $S$ has shape $\mu$ as in \eqref{eq:k2special-mu}, where in case (1) $s=2t$, $a=r'$, $b=r$, and in case (2) $s=2t+1$, $a=r$, $b=r'$. Since $\mu$ is $(k,2)$-special, we conclude that $\a$ is well-defined.

Conversely, given $S\in \mc{T}^{k,2}_{ss}$ we define $A$ to be the word obtained by concatenating the odd rows of $S$, and let $B$ be the word obtained by concatenating the even rows. Let $T$ be the $2$-row tableau with row words $A$ and $B$, and set $\b(S)=T$ once we have shown that $T$ is $k$-restricted. With notation \eqref{eq:k2special-mu}, if $s=2t$ then $T$ has shape $\ll=(tk+a,tk+b)$, and if $s=2t+1$ then $T$ has shape $\ll=((t+1)k+b,tk+a)$. Since $S$ has increasing columns, it follows that $T$ does as well. In order to prove that $T$ is semistandard it suffices to check, for each~$u$, the junction between row $u$ and row $u + 2$ of $S$, where we need to verify that $S_{(u,k)}\leq S_{(u+2,1)}$.

To that end we fix $u\leq s$ (where we assume that $b>0$ if $u=s$), write  $S_{(u,k)}=i+1$, and suppose by contradiction that $S_{(u+2,1)}\leq i$. If we write $\nu$ for the shape of $S[i]$, it follows that $\nu_u<k$ and $\nu_{u+2}>0$, and therefore $\nu$ is not $(k,2)$-special. This is a contradiction which proves that $S_{(u,k)}\leq S_{(u+2,1)}$, hence $T$ is semistandard. It remains to verify that $T$ is $k$-restricted.

If we assume that $T[i]$ has shape $\gamma=(\gamma_1,\gamma_2)$ then we have to prove that $T[i+1]$ has at most $\gamma_2+k$ columns. Suppose that $S[i]$ has shape $(k^{r},c,d)$, where $k>c\geq d\geq 0$. Since $S[i+1]$ is semistandard with at most $k$ columns, it follows that the entries of $S[i+1]\setminus S[i]$ can only be located at
\begin{itemize}
    \item $(r+1,c+1),(r+1,c+2),\cdots,(r+1,k)$,
    \item $(r+2,d+1),(r+2,d+2),\cdots,(r+2,c)$,
    \item $(r+3,1),(r+3,2),\cdots,(r+3,d)$.
\end{itemize}
If $r=2m$ is even, then $T[i+1]$ is obtained from $T[i]$ by adding at most $(k-c+d)$ entries to the first row, and at most $(c-d)$ entries to the second row. Since $\gamma=(mk+c,mk+d)$, this implies that $T[i+1]$ has at most 
\[\max(\gamma_1+k-c+d,\gamma_2+c-d)=\max((m+1)k+d,mk+c)=(m+1)k+d=\gamma_2+k\]
columns. If $r=2m+1$ is odd, then $T[i+1]$ is obtained from $T[i]$ by adding at most $(c-d)$ entries to the first row, and at most $(k-c+d)$ entries to the second row. We have $\gamma=((m+1)k+d,mk+c)$, so $T[i+1]$ has at most
\[\max(\gamma_1+c-d,\gamma_2+k-c+d)=\max((m+1)k+c,(m+1)k+d)=(m+1)k+c=\gamma_2+k\]
columns. This concludes the proof that $T$ is a $k$-restricted tableau, so $\beta$ is a well-defined transformation.

Since $\a,\b$ are inverse constructions and are weight-preserving, this concludes our proof.
\end{proof}

For $\ll=(\ll_1,\ll_2)$ with $\ll_1\leq\ll_2+k$ we write $\mc{T}^{k,2}_{res}(\ll;n)$ for the set of $k$-restricted tableaux of shape $\ll$ with entries in $\{1,\cdots,n\}$. Similarly, for $\mu$ as in \eqref{eq:k2special-mu} we write $\mc{T}^{k,2}_{ss}(\mu;n)$ for the set of $(k,2)$-semistandard tableaux with entries in $\{1,\cdots,n\}$. The construction in the proof of Lemma~\ref{lem:special-vs-restricted} induces bijections
\[ \a:\mc{T}^{k,2}_{res}(\ll;n)\lra\mc{T}^{k,2}_{ss}(\mu;n),\ \b:\mc{T}^{k,2}_{ss}(\mu;n)\lra \mc{T}^{k,2}_{res}(\ll;n),\text{ when }\ll_1=\mu_1+\mu_3+\cdots,\ \ll_2 = \mu_2+\mu_4+\cdots.\]
If $\ll$ is $k$-restricted and $\mu$ is $(k,2)$ special, and if they are related as above, then we also write $\a(\ll)=\mu$ and $\b(\mu)=\ll$. More explicitly, we have
\[ \b(\mu) = (\mu_1+\mu_3+\cdots,\mu_2+\mu_4+\cdots).\]
For $\ll=(\ll_1,\ll_2)$, if we write $\ll_2=tk+b$ with $0\leq b<k$, and $\ll_1=tk+a$ with $0\leq a \leq k+b$, then we have
\[ \a(\ll) = \begin{cases}
    (k^{2t},a,b) & \text{if }0\leq a\leq k, \\
    (k^{2t+1},b,a-k) & \text{if }k\leq a\leq k+b.
\end{cases}\]

We next consider the special case $k=p-2\geq 1$, and show that the graded components of $A_{c+1}(\ul{z})=\vartheta^{p-2,c}(\ul{z})$ are simple modular characters (recall that for $p=2$, $A_1(\ul{z})=1$). Given a partition $\mu$, we write $\ell_{\mu}(\ul{z})=\ell_{\mu}^p(\ul{z})$ for the character of the simple $\GL_n$-module of highest weight $\mu$ in characteristic $p$ (and make the convention that $\ell_{\mu}(\ul{z})=0$ if $\mu$ has more than $n$ parts). If we write $\ul{z}^S=z_1^{a_1}\cdots z_n^{a_n}$ where $\ul{a}$ is the weight of a tableau $S$, then it follows from \cite{mat-pap}*{Theorem~4.3} that
\begin{equation}\label{eq:charmu-p-2special}
    \ell_{\mu}(\ul{z}) = \sum_{S\in \mc{T}^{p-2,2}_{ss}(\mu;n)} \ul{z}^S,\qquad\text{ if $\mu$ is $(p-2,2)$-special}.
\end{equation}
Writing $\mu=\a(\ll)$ for some $(p-2)$-restricted partition, it follows that
\[\ell_{\a(\ll)}(\ul{z}) = \sum_{T\in \mc{T}^{p-2,2}_{res}(\ll;n)} \ul{z}^T.\]
Combining this with \eqref{eq:Vka1n} and \eqref{eq:Ac=vthetap-2c-1}, it follows that for each $c=0,\cdots,p-2$ we have
\begin{equation}\label{eq:Ac-from-lmu}
A_{c+1}(\ul{z})=\vartheta^{p-2,c}(\ul{z}) = \sum_{\substack{m\geq 0 \\ \ll=(m+c,m)}}\ell_{\a(\ll)}(\ul{z}),
\end{equation}
and therefore the non-zero graded components of the Verlinde polynomials at level $p-2$ are simple modular characters of the general linear group.

If we combine \eqref{eq:Thetaz=wp-1} with \eqref{eq:even-versym-kc-polys} and \eqref{eq:Vka1n}, then we obtain combinatorial descriptions for $\Theta(\ul{z})$ and $\Theta'(\ul{z})$ in terms of $(2p-2)$-restricted tableaux, and applying Lemma~\ref{lem:special-vs-restricted} we obtain equivalent formulations using $(2p-2,2)$-semistandard tableaux. Nevertheless, the results of \cite{mat-pap} no longer apply, and it is not apparent how to express $\Theta(\ul{z})$ and $\Theta'(\ul{z})$ as linear combinations of simple modular characters. Instead we can use the trigonometric interpretation to express $\Theta(\ul{z})$ and $\Theta'(\ul{z})$ in terms of elementary symmetric polynomials and the symmetric functions $A_c(\ul{z})$. More precisely, we have
\[
\Theta(\ul{z})=\frac{E(\ul{z})}{p}\cdot
\left(\Psi_p(\ul{z})+\sum_{c=1}^{p-1}(-1)^{c-1}\cdot(p-c)\cdot A_c(\ul{z})\right),\
\Theta'(\ul{z})=\frac{(-1)^{p+1}E(\ul{z})}{p}
\left(
\Psi_p(\ul{z})+\sum_{c=1}^{p-1}(-1)^c \cdot c\cdot A_c(\ul{z})
\right),
\]
where $\Psi_p(\ul{z})$ is as in \eqref{eq:def-thetar-Psir}. Indeed, noting that the term $r=p$ in \eqref{eq:def-Thetaz-Theta'z} contributes $\displaystyle\frac{E(\ul{z})\cdot \Psi_p(\ul{z})}{p}$ to $\Theta(\ul{z})$ and $\displaystyle\frac{(-1)^{p+1}E(\ul{z})\cdot \Psi_p(\ul{z})}{p}$ to $\Theta'(\ul{z})$, it suffices to identify the coefficients of $\Psi_r(\ul{z})$ for $r=1,\cdots,p-1$. The formulas for $\Theta(\ul{z})$ and $\Theta'(\ul{z})$ follow from the trigonometric identities
\[\frac{2}{p}\cdot\sum_{c=1}^{p-1}(-1)^{c-1}\cdot(p-c)\cdot\sin(c\theta_r)
=\tan\left(\frac{\theta_r}{2}\right),\qquad
\frac{2}{p}\cdot\sum_{c=1}^{p-1}(-1)^c\cdot c\cdot\sin(c\theta_r)
= (-1)^{r-p}\cdot\tan\left(\frac{\theta_r}{2}\right),\]
which in turn can be derived based on \cite{GRints}*{Section~1.34}, or by substituting $z=e^{-i\theta_r}$ into
\[\sum_{c=1}^{p-1}z^c = \frac{z-z^p}{1-z}\quad\text{and}\quad \sum_{c=1}^{p-1}c\cdot z^c = \frac{z-p\cdot z^{p}+(p-1)\cdot z^{p+1}}{(1-z)^2}.\]

\subsection{The reduced GHM ring as a tensor product of Verlinde algebras}
\label{subsec:redGHM=tensor-Verlinde}
We write $[a,b] = \{a,a+1,\cdots,b\}$, $[a,b]_p = \{i\in[a,b] : p\nmid i\}$, and for $q=p^s$ consider the involution
\begin{equation}\label{eq:def-iota} 
\iota_q : [1,q-1]_p \lra [1,q-1]_p,\quad\iota_q(c)=q-c.
\end{equation}
\begin{lemma}\label{lem:bij-pq-iota}
If $q=p^s$, $s\geq 1$, and $\iota=\iota_q$ is as in \eqref{eq:def-iota}, then there exists a bijection
\[ \gamma_q: [0,p-1] \times [1,q-1]_p \lra [1,pq-1]_p, \]
defined by $\gamma_q(c_0,c_1) = c_0q + \iota^{c_0}(c_1)$, where $\iota^{c_0}$ denotes the composition of $\iota$ with itself $c_0$ times.
\end{lemma}

\begin{proof}
    Since $0\leq c_0q\leq(p-1)q$ and $\iota^{c_0}(c_1)\in[1,q-1]_p$, it follows that $\gamma_q$ is well-defined. Moreover, the source and target of $\gamma_q$ have the same cardinality, namely $(p-1)q$, so it suffices to verify that $\gamma_q$ is injective. If we assume that $\gamma_q(a_0,a_1)=\gamma_q(b_0,b_1)$, then we have
    \[ (b_0-a_0)q = \iota^{a_0}(a_1) - \iota^{b_0}(b_1).\]
    Since $|\iota^{a_0}(a_1) - \iota^{b_0}(b_1)|<q$, this forces $b_0=a_0$ and $\iota^{a_0}(a_1) = \iota^{b_0}(b_1)$, which in turn implies $a_1=b_1$. It follows that $\gamma_q$ is injective, which concludes our proof.
\end{proof}

The next result follows from \cite{CEO}. We state the version needed here and include a proof for completeness.

\begin{theorem}[{\cite{CEO}*{Corollary~5.3}}]\label{thm:GHM-as-tensor-Verlinde}
    For $q=p^s$, $s\geq 1$, consider the subring $\ol{\Delta}^u_q$ of $\ol{\Delta}^u$ with $\bb{Z}$-basis $\delta_i$, $i\in[1,q-1]_p$. There exists an isomorphism of algebras
    \[ \gamma:\left(\mc{V}'_{p-1}\right)^{\oo (s-1)} \oo \mc{V}_{p-2} \lra \ol{\Delta}^u_q,\]
    which can be constructed inductively from the following:
    \begin{enumerate}
        \item If $s=1$ then $\gamma:\mc{V}_{p-2}\lra\ol{\Delta}^u_p$ is defined by $\gamma(\phi_i)=\delta_{i+1}$ for $i=0,\cdots,p-2$.
        \item For $q=p^s$, the bijection $\gamma_q$ from Lemma~\ref{lem:bij-pq-iota} induces a ring isomorphism, denoted by the same symbol,
        \[\gamma_q : \mc{V}'_{p-1} \oo \ol{\Delta}^u_q \lra \ol{\Delta}^u_{pq},\quad\text{defined by }\quad \gamma_q(\phi'_{c_0} \oo \delta_{c_1}) = \delta_{\gamma_q(c_0,c_1)}\quad\text{for $c_0\in[0,p-1]$ and $c_1\in[1,q-1]_p$}.\]
    \end{enumerate}
\end{theorem}

\begin{proof} In case (1), $\gamma(\phi_i)=\delta_{i+1}$ defines an isomorphism by \cite{renaud}*{Theorem~1}. In case (2), it follows from Lemma~\ref{lem:bij-pq-iota} that $\gamma_q$ is a $\bb{Z}$-module isomorphism, so we only have to verify that $\gamma_q$ respects multiplication.

Consider $a,b\in[1,pq-1]_p$, write $a=\gamma_q(a_0,a_1)$ and $b=\gamma_q(b_0,b_1)$, with $a_0,b_0\in[0,p-1]$, $a_1,b_1\in[1,q-1]_p$, and assume that $a_0\leq b_0$.
Working in the reduced ring, we write
\[ \delta_{a_1}\delta_{b_1} = \sum_{j=1}^\ell m_j\delta_{\mu_j},\]
where each $\mu_j$ is coprime to $p$. By \cite{renaud}*{G-2}, we have $\delta_{q-1}\delta_c = \delta_{\iota(c)}$ for $c\in[1,q-1]_p$. Multiplying the identity above by $\delta_{q-1}^{a_0+b_0}$ yields
\[ \delta_{\iota^{a_0}(a_1)}\delta_{\iota^{b_0}(b_1)} = \sum_{j=1}^{\ell} m_j\delta_{\iota^{a_0+b_0}(\mu_j)}.\]
It follows from \cite{renaud}*{Theorem~2} (or \cite{GPX}*{Proposition~6}) that
\[\delta_a\delta_b = \sum_{j=1}^{\ell}m_j\cdot\left(\sum_{k=0}^{\min(2a_0,2p-2-2b_0)} \delta_{(b_0-a_0+k)q+\iota^k(\mu_j)} \right).\]
Since $\iota$ is an involution, we have that $\iota^{b_0-a_0+k} \circ \iota^k = \iota^{a_0+b_0}$, hence we can apply $\gamma_q^{-1}$ to get
\[
\begin{aligned}
\gamma_q^{-1}(\delta_a\delta_b) &= \sum_{j=1}^{\ell}m_j\cdot\left(\sum_{k=0}^{\min(2a_0,2p-2-2b_0)} \phi'_{b_0-a_0+k} \oo \delta_{\iota^{a_0+b_0}(\mu_j)} \right) \\
&=\left(\sum_{k=0}^{\min(2a_0,2p-2-2b_0)} \phi'_{b_0-a_0+k}\right) \oo \left(\sum_{j=1}^{\ell}m_j\delta_{\iota^{a_0+b_0}(\mu_j)} \right) = \phi'_{a_0}\phi'_{b_0} \oo \delta_{\iota^{a_0}(a_1)}\delta_{\iota^{b_0}(b_1)} = \gamma_q^{-1}(\delta_a)\gamma_q^{-1}(\delta_b).
\end{aligned}
\]
This implies that $\gamma_q^{-1}$ is a ring isomorphism, hence so is $\gamma_q$.
\end{proof}

\subsection{Higher structure constants in the reduced GHM ring}
\label{subsec:high-strconst-GHM}

We consider the higher structure constants $m_c(\ul{a})=m_c(a_1,\cdots,a_n)$ for the ring $\ol{\Delta}^u$, defined by $m_c(\ul{a})=0$ if $p|(1+a_1)\cdots(1+a_n)\cdot c$, and by the equality
\[\delta_{1+a_1}\cdots\delta_{1+a_n} = \sum_{c\geq 1} m_c(\ul{a})\cdot\delta_c \quad\text{in }\ol{\Delta}^u.\]
Using Proposition~\ref{prop:delcj-in-product} and the notation \eqref{eq:fcat-small-c}, we have
\[ m_c(\ul{a})=m_{c,j}(\ul{a})\quad\text{and}\quad f^{\ul{a}}_c(t) = m_c(\ul{a})\cdot t^{j}\quad\text{ for }j=\frac{a_1+\cdots+a_n+1-c}{2}.\]
For $p\nmid c$, we define the power series $\bN_c(t)\in\Lambda[[t]]$ via
\begin{equation}\label{eq:def-Nct}
    \bN_c(t) = \sum_{a_1,\cdots,a_n\geq 0} m_c(\ul{a})\cdot (tz_1)^{a_1}\cdots(tz_n)^{a_n},
\end{equation}
and proceed to show that for $c=1,\cdots,p-1$, it can be computed by \eqref{eq:Nct=AM10product}. To that end we set for $q=p^s$
\[    \bN_c^q(t) = \sum_{a_1,\cdots,a_n=0}^{q-2} m_c(\ul{a})\cdot (tz_1)^{a_1}\cdots(tz_n)^{a_n},
\text{
and note that }\bN_c(t) = \lim_{q\to\infty}\bN_c^q(t).\]

\begin{proposition}\label{prop:recursive-Ncqt}
    Suppose that $1\leq c\leq p-1$, write $q=p^s$, and use notation \eqref{eq:def-Mz-Mt}, \eqref{eq:def-Acz}. We have $\bN_c^p(t)=A_c(t\ul{z})$, and:
    \begin{itemize}
        \item If $p=2$ then $\bN_c^{2q}(t)=\bN_c^{q}(t)\cdot F^q(\Theta(t\ul{z}))$.
        \item If $p>2$ then
    \begin{equation}\label{eq:Npqc-Npq-c=product}
    \begin{bmatrix}
    \bN_c^{pq}(t) & \bN_{pq-c}^{pq}(t)
\end{bmatrix} = \begin{bmatrix}
    \bN_c^{q}(t) & \bN_{q-c}^{q}(t)
\end{bmatrix}\cdot
F^{q}(B(t\ul{z})).
\end{equation}
    \end{itemize}
In particular, the equality \eqref{eq:Nct=AM10product} holds.
\end{proposition}

\begin{proof}
 The identification $\bN_c^p(t)=A_c(t\ul{z})$ follows from \eqref{eq:Ac=vthetap-2c-1}, and the isomorphism $\mc{V}_{p-2}\simeq\ol{\Delta}^u_p$ in Theorem~\ref{thm:GHM-as-tensor-Verlinde}(1). Noting that the variable $t$ in $\bN_c^q(t)$ only keeps track of the degree with respect to $z_1,\cdots,z_n$, we set $t=1$ and write $\bN^q_c(1) = N^q_c(\ul{z})\in\Lambda$, so that $\bN^q_c(t)=N^q_c(t\ul{z})$. We then have to verify that
 \begin{equation}\label{eq:Npqcz-first}
 N^{pq}_c(\ul{z}) = F^q(\Theta_{even}(\ul{z}))\cdot N^q_c(\ul{z}) + F^q(\Theta_{odd}(\ul{z}))\cdot N^q_{q-c}(\ul{z})\quad\text{ for all }p, 
\end{equation}
\begin{equation}\label{eq:Npqcz-second}
N^{pq}_{pq-c}(\ul{z}) = F^q(\Theta'_{odd}(\ul{z}))\cdot N^q_c(\ul{z}) + F^q(\Theta'_{even}(\ul{z}))\cdot N^q_{q-c}(\ul{z})\quad\text{ for }p>2,
\end{equation}
which we do by employing Theorem~\ref{thm:GHM-as-tensor-Verlinde}.
 
 We define for each $q=p^s$ a polynomial $N^q(\ul{z})$ with coefficients in $\ol{\Delta}^u_q$, given by
 \[ N^q(\ul{z}) = \sum_{a_1,\cdots,a_n=0}^{q-2}\delta_{1+a_1}\cdots\delta_{1+a_n}\cdot z_1^{a_1}\cdots z_n^{a_n} = \sum_{c=1}^{q-1}N^q_c(\ul{z})\cdot\delta_c.\]
 We also define polynomials $Q(\ul{z})$, $Q_{even}(\ul{z})$, $Q_{odd}(\ul{z})$, with coefficients in $\mc{V}'_{p-1}$, by
 \[Q(\ul{z}) = \sum_{m_1,\cdots,m_n=0}^{p-1} \phi'_{m_1}\cdots\phi'_{m_n}\cdot z_1^{m_1}\cdots z_n^{m_n} = Q_{even}(\ul{z}) + Q_{odd}(\ul{z}),\]
 where $Q_{even}(\ul{z})$ collects the even-degree terms of $Q(\ul{z})$, and $Q_{odd}(\ul{z})$ collects the odd-degree terms. The isomorphism $\gamma_q$ in Theorem~\ref{thm:GHM-as-tensor-Verlinde} extends to the respective polynomial rings by sending $\gamma_q(z_i)=z_i$, and we get
 \[\gamma_q^{-1}(N^{pq}(\ul{z})) = \sum_{a_1,\cdots,a_n=0}^{pq-2} \gamma_q^{-1}(\delta_{1+a_1}\cdots\delta_{1+a_n})\cdot z_1^{a_1}\cdots z_n^{a_n} = \sum_{c=1}^{pq-1}N_c^{pq}(\ul{z})\cdot\gamma_q^{-1}(\delta_c).\]
 For each $a_i$ with $\delta_{1+a_i}\neq 0$, we can write $a_i=qm_i+b_i$ with $0\leq m_i\leq p-1$, $0\leq b_i\leq q-2$. It follows that $1+a_i = \gamma_q(m_i,\iota^{m_i}(1+b_i))$, where $\iota=\iota_q$ is as in \eqref{eq:def-iota}. Using the identity $\delta_{\iota^{m_i}(1+b_i)} = \delta_{q-1}^{m_i} \cdot \delta_{1+b_i}$, we get
 \begin{equation}\label{eq:gammaqinv-Nqz}
 \begin{aligned}
 \gamma_q^{-1}(N^{pq}(\ul{z})) &= \sum_{m_1,\cdots,m_n=0}^{p-1}\phi'_{m_1}\cdots\phi'_{m_n} z_1^{qm_1}\cdots z_n^{qm_n} \oo \left(\delta_{q-1}^{m_1+\cdots+m_n}\sum_{b_1,\cdots,b_n=0}^{q-2} \delta_{1+b_1}\cdots\delta_{1+b_n} z_1^{b_1}\cdots z_n^{b_n}\right) \\
 &= F^q(Q_{even}(\ul{z})) \oo N^q(\ul{z}) + F^q(Q_{odd}(\ul{z})) \oo \delta_{q-1}N^q(\ul{z})
 \end{aligned}
 \end{equation}
 
 We can now verify \eqref{eq:Npqcz-first}. For $c=1,\cdots,p-1$ we have $\gamma_q^{-1}(\delta_c) = \phi'_0 \oo \delta_c$, hence $N^{pq}_c(\ul{z})$ is the coefficient of $\phi'_0\oo \delta_c$ in $\gamma_q^{-1}(N^{pq}(\ul{z}))$. It follows from \eqref{eq:Thetaz=wp-1} that $\Theta_{even}(\ul{z})$ (respectively $\Theta_{odd}(\ul{z})$) is the coefficient of $\phi'_0$ in $Q_{even}(\ul{z})$ (respectively in $Q_{odd}(\ul{z})$). Moreover, $N_c^q(\ul{z})$ (respectively $N_{q-c}^q(\ul{z})$) is the coefficient of $\delta_c$ in $N^q(\ul{z})$ (respectively $\delta_{q-1}N^q(\ul{z})$), hence \eqref{eq:Npqcz-first} follows from \eqref{eq:gammaqinv-Nqz}.

 To prove \eqref{eq:Npqcz-second}, we write $\gamma_q^{-1}(\delta_{pq-c}) = \phi'_{p-1}\oo \iota^{p-1}(\delta_{q-c}) = \phi'_{p-1}\oo \delta_{q-c}$, using the fact that $p-1$ is even. It follows that $N^{pq}_{pq-c}(\ul{z})$ is the coefficient of $\phi'_{p-1}\oo \delta_{q-c}$ in $\gamma_q^{-1}(N^{pq}(\ul{z}))$. The second equality in \eqref{eq:Thetaz=wp-1} implies that $\Theta'_{even}(\ul{z})$ (respectively $\Theta'_{odd}(\ul{z})$) is the coefficient of $\phi'_{p-1}$ in $Q_{even}(\ul{z})$ (respectively in $Q_{odd}(\ul{z})$), hence \eqref{eq:Npqcz-second} follows as before from \eqref{eq:gammaqinv-Nqz}.

 Using \eqref{eq:Thetaz-p=2}, the identity $\bN_c^{2q}(t)=\bN_c^{q}(t)\cdot F^q(\Theta(t\ul{z}))$ for $p=2$ becomes a special case of \eqref{eq:Npqcz-first}, while \eqref{eq:Npqc-Npq-c=product} for $p>2$ follows by combining \eqref{eq:Npqcz-first} and \eqref{eq:Npqcz-second}. 
 
 The equality \eqref{eq:Nct=AM10product} for $p>2$ follows by iterating \eqref{eq:Npqc-Npq-c=product}. 
 When $p=2$ we only have $c=1$ in \eqref{eq:Nct=AM10product}, and moreover $A_1(t\ul{z})=1$. We recall from \eqref{eq:Mz-charp=2} that $B(\ul{z})$ is an upper triangular matrix, hence so is $\bB(t)$, and \eqref{eq:Nct=AM10product} can be rewritten as
 \[\begin{bmatrix}
    A_c(t\ul{z}) & A_{p-c}(t\ul{z})
\end{bmatrix} \cdot \bB(t)\cdot\begin{bmatrix} 1 \\ 0 \end{bmatrix} = \begin{bmatrix}
    1 & 1
\end{bmatrix} \cdot \begin{bmatrix} \prod_{s\geq 1}F^{2^s}(\Theta(t\ul{z})) \\ 0 \end{bmatrix} = \prod_{s\geq 1}F^{2^s}(\Theta(t\ul{z})) = \bN_1(t),\]
which follows by iterating the identity $\bN_1^{2q}(t)=\bN_1^{q}(t)\cdot F^q(\Theta(t\ul{z}))$.
\end{proof}

\section{Cohomology characters}
\label{sec:coh-characters}

We can now put together all the results developed so far to prove the main theorems on cohomology characters discussed in the Introduction.

\begin{proof}[Proof of Theorem~\ref{thm:main-nonrec-coh}]
Recalling \eqref{eq:def-Guv} we can write
\[\begin{aligned}
\bG(u,v) &= \sum_{d\geq 0,\ e\geq -1} h^{n-1}(\mc{O}_X(e+1,-d+1-n))\cdot z_1^{-1}\cdots z_n^{-1} u^d v^{d+e} \\
&=\sum_{a_1,\cdots,a_n} u\cdot P_{\ul{a}}(u^{-1}) \cdot (uvz_1)^{a_1}\cdots(uvz_n)^{a_n}.
\end{aligned}\]
where the second equality follows from \eqref{eq:Pat} and \eqref{eq:H01=Hn-2n-1}. If we set $t=u^{-1}$ and apply Theorem~\ref{thm:Pa-recursion}, it follows that we can write
\begin{equation}\label{eq:Guv-times-powsu}
\bG(u,v)(1+u+\cdots+u^{p-1}) = u^{p}(\Sigma_1 + \Sigma_2),
\end{equation}
where
\[\Sigma_1 = \sum_{a_1,\cdots,a_n}\sum_{0\leq i_1,\cdots,i_n\leq p-1} P_{\ul{a}'}(t^p)\cdot t^{i_1+\cdots+i_n}\cdot (uvz_1)^{a_1}\cdots(uvz_n)^{a_n}\]
where $\ul{a}'$ is defined by \eqref{eq:def-aprs}, and
\[ \Sigma_2 = \sum_{a_1,\cdots,a_n}\sum_{c=1}^{p-1}f^{\ul{a}}_c(t) \cdot (t^c+t^{c+1}+\cdots+t^{p-1})\cdot (uvz_1)^{a_1}\cdots(uvz_n)^{a_n}.\]

To compute $u^p\cdot\Sigma_1$, we note that conditions \eqref{eq:def-aprs} can be rewritten as
\[ a_s = i_s + pa'_s + j_s,\quad\text{for some }0\leq j_s\leq p-1,\]
and that with this notation we have
\[ t^{i_s}\cdot(uvz_s)^{a_s} = (vz_s)^{i_s}\cdot(uvz_s)^{j_s}\cdot F^p\left((uvz_s)^{a'_s}\right).\]
We can then reindex the summation in $\Sigma_1$ using the parameters $i_s,j_s,a'_s$ to get
\[
\begin{aligned}
    u^p\cdot\Sigma_1 &= \sum_{\substack{a'_1,\cdots,a'_n \\ 0\leq i_1,\cdots,i_n\leq p-1 \\ 0\leq j_1,\cdots,j_n\leq p-1}} F^p\left(u\cdot P_{\ul{a}'}(t)\cdot(uvz_1)^{a'_1}\cdots(uvz_n)^{a'_n} \right) \cdot \prod_{s=1}^n \left((vz_s)^{i_s}\cdot(uvz_s)^{j_s}\right) \\
    &= F^p\left(\bG(u,v)\right)\cdot\bh^{(p)}(v)\cdot \bh^{(p)}(uv).
\end{aligned}
\]

Recalling \eqref{eq:fcat-small-c} and using the identity
\[u^p\cdot t^{\frac{a_1+\cdots+a_n+1-c}{2}} \cdot (t^c+t^{c+1}+\cdots+t^{p-1}) =  u^{1/2}\cdot\frac{u^{c/2}-u^{p-c/2}}{1-u}\cdot u^{\frac{-1}{2}(a_1+\cdots+a_n)} \]
we obtain
\[ 
\begin{aligned}
u^p\cdot\Sigma_2 &= \sum_{c=1}^{p-1}u^{1/2}\cdot\frac{u^{c/2}-u^{p-c/2}}{1-u}\cdot\left(\sum_{a_1,\cdots,a_n}m_c(\ul{a})\cdot (u^{1/2}vz_1)^{a_1}\cdots(u^{1/2}vz_n)^{a_n} \right) \\
&=\sum_{c=1}^{p-1}u^{1/2}\cdot\frac{u^{c/2}-u^{p-c/2}}{1-u}\cdot\bN_c(u^{1/2}v).
\end{aligned}
\]
The first part of Theorem~\ref{thm:main-nonrec-coh} now follows from \eqref{eq:Guv-times-powsu} after dividing by $1+u+\cdots+u^{p-1}$.

To prove the second part of the theorem, we first note that if $q=pq'$ then
    \[ \bh^{(q)} = \bh^{(p)}\cdot F^p(\bh^{(p)})\cdots F^{q'}(\bh^{(p)}).\]
    Similarly, we have
    \[ 1 + u + \cdots + u^{q-1} = (1+u+\cdots+u^{p-1})\cdot F^p(1+u+\cdots+u^{p-1})\cdots F^{q'}(1+u+\cdots+u^{p-1}).\]
    It follows that if we define $\bg$ via
    \[ \bg(u,v) = \frac{\bh^{(p)}(uv)\cdot \bh^{(p)}(v)}{1+u+\cdots+u^{p-1}}\]
    then
    \begin{equation}\label{eq:prod-Fqguv}
    \bg(u,v)\cdot F^p(\bg(u,v))\cdots F^{q'}(\bg(u,v)) = \frac{\bh^{(q)}(uv)\cdot \bh^{(q)}(v)}{1 + u + \cdots + u^{q-1}}. 
    \end{equation}
    We can now rewrite the functional equation we obtained in the first part of the theorem to get
    \[ 
    \begin{aligned}
    \bG(u,v) &= \bN(u,v) + \bg(u,v)\cdot F^p(\bG(u,v)) \\
      &= \bN(u,v) + \bg(u,v)\cdot F^p(\bN(u,v)) + \bg\cdot F^p(\bg(u,v))\cdot F^{p^2}(\bN(u,v)) +\cdots \\
      &= \sum_{q\geq 1} \frac{\bh^{(q)}(uv)\cdot \bh^{(q)}(v)}{1+u+\cdots+u^{q-1}}\cdot F^q\left(\bN(u,v)\right),
    \end{aligned}
    \]
    where the last equality follows from \eqref{eq:prod-Fqguv} and concludes our proof.
\end{proof}

We now specialize further the results of Theorem~\ref{thm:main-nonrec-coh} to the case of characteristic $p=2$.

\begin{proof}[Proof of Theorem~\ref{thm:h1coh-char2-nonrecursive}]
    It follows from Corollary~\ref{cor:prod-dela-odd-Nim} and equation \eqref{eq:def-Nct} that
    \[\bN_1(t) = \sum_{a_1\oplus\cdots\oplus a_n=0} (tz_1)^{2a_1}\cdots (tz_n)^{2a_n} \overset{\eqref{eq:def-Nim-pol}}{=}\sum_{m\geq 0} F^2(\mc{N}_m(\ul{z})) \cdot t^{4m}.\]
    If we write $\mc{N}_m=\mc{N}_m(\ul{z})$, we get from the above equality combined with \eqref{eq:N1t-char2} that
    \[\bN(u,v) = \frac{u}{1+u}\cdot\sum_{m\geq 0} F^2(\mc{N}_m(\ul{z})) \cdot u^{2m}\cdot v^{4m}.\]
    We can then write for $q=2^s$
    \[
    \begin{aligned}
        &\frac{\bh^{(q)}(uv)\cdot \bh^{(q)}(v)}{1+u+\cdots+u^{q-1}}\cdot F^q\left(\bN(u,v)\right) = \frac{\bh^{(q)}(uv)\cdot \bh^{(q)}(v)\cdot u^q}{1 + u + \cdots + u^{2q-1}} \cdot \sum_{m\geq 0} F^{2q}(\mc{N}_m)\cdot u^{2qm}\cdot v^{4qm} \\
        &= \bh^{(q)}(uv)\cdot \bh^{(q)}(v)\cdot(1-u)\cdot\left(\sum_{j\geq 0} u^{(2j+1)q}\right)\cdot \left(\sum_{m\geq 0} F^{2q}(\mc{N}_m)\cdot u^{2qm}\cdot v^{4qm}\right)
    \end{aligned}
    \]
    If we write $c_{a,b}$ for the coefficient of $u^av^b$ in $\bh^{(q)}(uv)\cdot \bh^{(q)}(v)\cdot(1-u)$ then we have for $b\geq 2a-1$ that
    \[ c_{a,b} = h^{(q)}_{b-a}\cdot h^{(q)}_a-h^{(q)}_{b-a+1}\cdot h^{(q)}_{a-1} = s^{(q)}_{(b-a,a)}.\]
    It follows from Theorem~\ref{thm:main-nonrec-coh} that we can rewrite
    \[ \bG(u,v) = \frac{u}{1+u}\cdot \bN_1(u^{1/2}v) + \sum_{\substack{q=2^s\geq 2 \\ a,b,m,j\geq 0}} F^{2q}(\mc{N}_m)\cdot c_{a,b} \cdot u^{a+(2m+2j+1)q}\cdot v^{4mq+b}.\]
    To compute the coefficient of $u^d v^{d+e}$ for $e\geq d-1$ we set $a=d-(2m+2j+1)q$, $b=d+e-4mq$ and note that the inequality $e\geq d-1$ forces $b\geq 2a-1$. Moreover, for $e\geq d-1$, the coefficient of $u^dv^{d+e}$ in $\frac{u}{1+u}\cdot \bN_1(u^{1/2}v)$ is $0$. Using the fact that $b-a=e-(2m-2j-1)q$ it follows that the coefficient of $u^d v^{d+e}$ in $\bG(u,v)$ is
    \[h^1(D^d\mc{R}(e)) = \sum_{\substack{q=2^s\geq 2 \\ m,j\geq 0}} F^{2q}(\mc{N}_m)\cdot s^{(q)}_{(e-(2m-2j-1)q,d-(2m+2j+1)q)}. \]
    Restricting further to parameters $m,j,q$ for which $d\geq (2m+2j+1)q$ (since otherwise the truncated Schur polynomial factor vanishes), we obtain the desired conclusion.
\end{proof}

\begin{proof}[Proof of Theorem~\ref{thm:coh-recursion}] 
We assume that $e\geq d-1$, and verify first the formula for $h^1(D^d\mc{R}(e))$ by equating the coefficients of $u^d v^{d+e}$ on the two sides of the equality \eqref{eq:func-eqn-Guv}. To that end, we note that \eqref{eq:def-Nuv} implies that if $u^d v^{d+e}$ appears as a term in $\bN(u,v)$, then $d\geq 1+\frac{d+e}{2}$, or equivalently $e\leq d-2$, which is contrary to our assumption. It then follows from \eqref{eq:func-eqn-Guv} that
\[h^1(D^d\mc{R}(e)) = \sum_{a,b} \Phi_{d-ap,e-bp}\cdot F^p\left(h^1(D^a\mc{R}(b))\right)\]
where $\Phi_{d,e}$ is the coefficient of $u^d v^{d+e}$ in 
\[\bg(u,v)=\frac{\bh^{(p)}(uv)\cdot \bh^{(p)}(v)}{1+u+\cdots+u^{p-1}}.\]
To conclude, we need to verify that $\Phi_{d,e}$ can be computed by \eqref{eq:def-Phi-de}. We have that $\Phi_{d,e}-\Phi_{d-p,e+p}$ is the coefficient of $u^d v^{d+e}$ in $\bg(u,v)\cdot(1-u^p) = \bh^{(p)}(uv)\cdot \bh^{(p)}(v)\cdot(1-u)$. As in the proof of Theorem~\ref{thm:h1coh-char2-nonrecursive}, this in turn equals $c_{d,d+e}=s^{(p)}_{(e,d)}$. Since $\Phi_{d-jp,e+jp}=0=s^{(p)}_{(e+jp,d-jp)}$ for $j\gg 0$, we obtain that
\[ \Phi_{d,e} = \sum_{j\geq 0} \left(\Phi_{d-jp,e+jp}-\Phi_{d-(j+1)p,e+(j+1)p}\right) = \sum_{j\geq 0}s^{(p)}_{(e+jp,d-jp)},\]
which establishes \eqref{eq:def-Phi-de} and proves the desired formula for $h^1(D^d\mc{R}(e))$.

To prove the corresponding result for $h^0(D^d\mc{R}(e))$ we use \eqref{eq:def-hd-sab}, \eqref{eq:chi-DdRe=sed}, and reduce to proving the identity
\[s_{(e,d)} = \sum_{a,b}\Phi_{d-ap,e-bp}\cdot F^p(s_{(b,a)}).\]
Multiplying by $u^d v^{d+e}$ and summing over all $d,e$, this is equivalent to
\[ \sum_{d,e} s_{(e,d)}u^d v^{d+e} = \bg(u,v)\cdot F^p\left( \sum_{a,b}  s_{(b,a)}u^a v^{a+b}\right).\]
Multiplying both sides by $1/(1-u)=1+u+u^2+\cdots$ and using the identity $\sum_{j\geq 0}s_{(e+j,d-j)}=h_e\cdot h_d$, it is then enough to verify that
\[\sum_{d,e} h_e\cdot h_d \cdot u^d v^{d+e} = h^{(p)}(uv)\cdot h^{(p)}(v)\cdot F^p\left(\sum_{a,b} h_b\cdot h_a\cdot u^a v^{a+b} \right),\]
which in turn reduces to proving that
\[\sum_{d}h_d\cdot t^d = h^{(p)}(t) \cdot F^p\left(\sum_{d}h_d\cdot t^d\right).\]
Since $\sum_d h_d\cdot t^d = \prod_{i=1}^n\frac{1}{1-tz_i}$, the above identity is a reformulation of \eqref{eq:def-bhqt} with $q=p$.
\end{proof}

We conclude this section with a proof of \cite{GRV}*{Conjecture~5.1}, which we restate as follows.

\begin{theorem}\label{thm:small-weights}
 Suppose that $tp\leq d<(t+1)p$, where $1\leq t<p$. We have for all $e\geq d-1$
 \[ h^1(D^d\mc{R}(e)) = \sum_{\substack{1\leq B\leq A\leq t \\ 0\leq J\leq A-B}} \left(h^{(p)}_{e+(B-J)p}\cdot h^{(p)}_{d-Ap} - h^{(p)}_{e+(B-J)p+1}\cdot h^{(p)}_{d-Ap-1}\right) \cdot F^p(s_{(A-B,J)}).\]
\end{theorem}

\begin{proof} It follows from Theorem~\ref{thm:coh-recursion} that
\[ h^1(D^d\mc{R}(e)) = \sum_{a=0}^t\sum_{b\geq -1}\sum_{j\geq 0}\left(h^{(p)}_{e-bp+jp}\cdot h^{(p)}_{d-ap-jp} - h^{(p)}_{e-bp+1+jp}\cdot h^{(p)}_{d-ap-1-jp}\right)\cdot F^p\left(h^1(D^a\mc{R}(b)\right).\]
Since $a<p$, we have as in characteristic zero
\[ h^1(D^a\mc{R}(b)) = \begin{cases}
    s_{(a-1,b+1)} & \text{if }-1\leq b\leq a-2, \\
    0 & \text{if }b\geq a-1.
\end{cases}\]
We can therefore restrict the above sum to $a\geq 1$, $-1\leq b\leq a-2$, and also restrict to $a+j\leq t$ (since otherwise $h^{(p)}_{d-ap-jp}=h^{(p)}_{d-ap-1-jp}=0$). If we set $A=a+j$, $B=j+1$, $J=b+1$, then $A-B = a-1 \geq b+1 = J \geq 0$ and $t\geq A\geq B\geq 1$. Moreover, we have
\[e+(B-J)p = e - bp + jp,\quad\text{and}\quad d-Ap = d-ap-jp.\]
It follows that the conclusion of Theorem~\ref{thm:coh-recursion} agrees in this case with our desired description of $h^1(D^d\mc{R}(e))$.
\end{proof}

\section*{Acknowledgements}
Experiments with Macaulay2 \cite{GS} and discussions with ChatGPT \cite{ChatGPT} have provided many valuable insights. The authors would like to thank Holger Brenner, Zhao Gao, Evan O'Dorney, Jenny Kenkel, Feliks R\k{a}czka, Alessio Sammartano, Anurag Singh, Kevin Tucker, and Keller VandeBogert for helpful discussions regarding various aspects of this project. Marangone gratefully acknowledges that this research was supported in part by the Pacific Institute for the Mathematical Sciences. Raicu and Reed acknowledge the support of the National Science Foundation Grant DMS-2302341.  Reed also acknowledges support by by the National Natural Science Foundation of China (No. 12288201). Part of the material in this paper is based upon work supported by the National Science Foundation under Grant No. DMS-1928930 and by the Alfred P. Sloan Foundation under grant G-2021-16778, while Raicu and Reed were in residence at the Simons Laufer Mathematical Sciences Institute (formerly MSRI) in Berkeley, California, during the Spring 2024 semester. Preliminary work on this project was done during the \emph{Pragmatic Research school in Algebraic Geometry and Commutative Algebra} at the University of Catania in June 2023, when all the authors were in residence -- we thank Francesco Russo for his hospitality, and for providing us with the opportunity to participate. Kyomuhangi acknowledges support of EMS committee for developing countries (EMS-CDC) towards travel to University of Catania from where this project emanated.

\end{document}